%% file: NIMCVFF.tex
\documentclass[a4paper,twoside]{amsart}
\usepackage{amssymb} 
\usepackage[all]{xy} 
\usepackage{color} 
\usepackage{lastpage} 
\usepackage[pdfstartview=FitH]{hyperref} 

\newtheorem{thm}{Theorem}[section]
\newtheorem{thma}{Theorem}
\newtheorem{cor}[thm]{Corollary}
\newtheorem{lem}[thm]{Lemma}

\newtheorem{prop}[thm]{Proposition}

\theoremstyle{definition}
\newtheorem{defn}[thm]{Definition}
\newtheorem{defna}[thma]{Definition}
\theoremstyle{remark}
\newtheorem{rem}[thm]{Remark}
\newtheorem{rema}[thma]{Remark}

\numberwithin{equation}{section}
\newcommand*{\cat}[1]{\mathbf{#1}} 
\newcommand*{\mto}{\rightarrow} 
\newcommand*{\cto}{\rightarrowtail} 
\newcommand*{\qto}{\twoheadrightarrow} 
\newcommand*{\wto}{\xrightarrow{\sim}} 
\newcommand*{\pair}[2]{\left\langle #1,#2 \right\rangle} 
\newcommand*{\com}[2]{\left[#1,#2\right]} 
\newcommand*{\set}[1]{\left\{#1\right\}} 
\newcommand*{\del}{\partial} 
\newcommand*{\D}{\mathcal{D}} 
\newcommand*{\hg}{\pi} 
\newcommand*{\comp}{\circ} 
\newcommand*{\Int}{\mathbb{Z}} 
\newcommand*{\FF}{\mathbb{F}} 
\newcommand*{\cmplx}[1]{{#1}^\bullet} 
\newcommand*{\tensor}{\otimes} 
\newcommand*{\isomorph}{\cong} 
\newcommand*{\op}{{o\!p}} 
\newcommand*{\et}{{\acute{e}\!t}} 
\newcommand*{\cont}{{c\!\!\:o\!n\!t}} 
\newcommand*{\algc}[1]{\overline{#1}} 
\newcommand*{\sheaf}[1]{\mathcal{#1}} 
\newcommand*{\oimm}
{\mathrel{\text{$\hookrightarrow\mkern-20mu\circ\mkern+6mu$}}} 
\newcommand{\openideals}{\mathfrak{I}} 
\newcommand{\NS}{\mathfrak{N}} 
\newcommand{\fundG}{\pi_1^{\et}} 
\newcommand{\Sect}{\Gamma} 
\newcommand{\Sectc}{\Gamma\!_c} 
\newcommand{\ncL}{\mathcal{L}} 

\newcommand{\ringtransf}{\Psi} 
\newcommand{\id}{\mathrm{id}} 
\newcommand{\Frob}{\mathfrak{F}} 

\DeclareMathOperator{\Fun}{\cat{Fun}} 
\DeclareMathOperator{\Cyl}{Cyl} 
\DeclareMathOperator{\Cone}{Cone} 
\DeclareMathOperator{\HF}{H} 
\DeclareMathOperator{\Hom}{Hom} 
\DeclareMathOperator{\Jac}{Jac} 
\DeclareMathOperator{\im}{im} 
\DeclareMathOperator{\coker}{coker} 
\DeclareMathOperator{\Spec}{Spec} 
\DeclareMathOperator{\God}{G} 
\DeclareMathOperator{\RDer}{R} 
\DeclareMathOperator{\Gal}{Gal} 
\DeclareMathOperator{\KTh}{K} 

\newdir{ >}{{}*!/-5pt/@{>}} 
\newdir{O}{{}*-\cir<2.5pt>{}} 
\xyoption{2cell} 
\UseTwocells %
\xyoption{poly}
\newcommand*{\lbl}[1]{\save+<-1pc,0pc>*\hbox{$\scriptstyle{#1}\quad$}\restore}
\newcommand*{\lbu}[1]{\save+<0pc,1pc>*\hbox{$\scriptstyle{#1}$}\restore} 

\allowdisplaybreaks[1]

\begin{document}

\title[Noncommutative Iwasawa Main Conjecture]{On a Noncommutative Iwasawa Main Conjecture for Varieties over Finite
Fields}
\author{Malte Witte}%

\address{Malte Witte\newline Ruprecht-Karls-Universit\"at Heidelberg\newline
Mathematisches Institut\newline
Im Neuenheimer Feld 288\newline
D-69120 Heidelberg }%
\email{witte@mathi.uni-heidelberg.de}%


\date{\today}%

\begin{abstract}
We formulate and prove an analogue of the noncommutative Iwasawa main conjecture for $\ell$-adic Lie extensions of a separated
scheme $X$ of finite type over a finite field of characteristic prime to $\ell$.
\end{abstract}

\maketitle
\input{BodyV2}

\bibliographystyle{amsalpha}
\bibliography{Literature}
\end{document}

%% file: BodyV2.tex
\section{Introduction}

In \cite{CFKSV}, Coates, Fukaya, Kato, Sujatha and Venjakob
formulate a noncommutative Iwasawa main conjecture for $\ell$-adic Lie
extensions of number fields. Other, partly more general versions are formulated in \cite{HK:EBKC+IMC}, \cite{RW:EquivIwaTh2}, and \cite{FK:CNCIT}. Following the approach of \cite{FK:CNCIT}, we formulate and prove below an analogous statement for $\ell$-adic Lie extensions of a separated scheme $X$ of finite type over a finite field $\FF_q$ with $q$ elements, where $\ell$ does not divide $q$.

Assume for the moment that $X$ is geometrically connected and let $G$ be a factor group of the fundamental group of $X$ such that $G\isomorph H\rtimes \Gamma$ where $H$ is a compact $\ell$-adic Lie group and $\Gamma=\Gal(\FF_{q^{\ell^{\infty}}}/\FF_q)\isomorph\Int_{\ell}$.
We write
$$
\Int_{\ell}[[G]]=\varprojlim \Int_{\ell}[G/U]
$$
for the Iwasawa algebra of $G$. Let
$$
S=\{f\in \Int_{\ell}[[G]]\colon \text{$\Int_{\ell}[[G]]/\Int_{\ell}[[G]]f$ is finitely generated as $\Int_{\ell}[[H]]$-module}\}
$$
denote Venjakob's canonical Ore set and write $\Int_{\ell}[[G]]_S$ for the localisation of $\Int_{\ell}[[G]]$ at $S$. We turn $\Int_{\ell}[[G]]$ into a smooth $\Int_{\ell}[[G]]$-sheaf $\sheaf{M}(G)$ on $X$ by letting the fundamental group of $X$ act contragrediently on $\Int_{\ell}[[G]]$.

Recall that there exists an exact localisation sequence of algebraic $\KTh$-groups
$$
\KTh_1(\Int_{\ell}[[G]])\mto\KTh_1(\Int_{\ell}[[G]]_S)\xrightarrow{d}\KTh_0(\Int_{\ell}[[G]],\Int_{\ell}[[G]]_S)\mto 0.
$$
Any endomorphism of perfect complexes of $\Int_{\ell}[[G]]_S$-modules which is a quasi-iso\-morphism gives rise to an element in the group $\KTh_1(\Int_{\ell}[[G]]_S)$. A system of generators for the relative $\KTh$-group $\KTh_0(\Int_{\ell}[[G]],\Int_{\ell}[[G]]_S)$ is given by perfect complexes of $\Int_{\ell}[[G]]$-modules whose cohomology groups are $S$-torsion.

For every continuous $\Int_{\ell}$-representation $\rho$ of $G$, there exists a homomorphism
$$
\rho\colon \KTh_1(\Int_{\ell}[[G]]_S)\mto Q(\Int_{\ell}[[\Gamma]])^{\times}
$$
into the units $Q(\Int_{\ell}[[\Gamma]])^{\times}$ of the field of fractions of $\Int_{\ell}[[\Gamma]]$. It is induced by sending $g\in G$ to $\det ([g]\rho(g)^{-1})$, with $[g]$ denoting the image of $g$ in $\Gamma$. On the other hand,  $\rho$ gives rise to a flat and smooth $\Int_{\ell}$-sheaf $\sheaf{M}(\rho)$ on $X$.

Let $\RDer\Sectc(X,\sheaf{F})$ and $\RDer\Sectc(\algc{X},\sheaf{F})$ be the compact cohomology of a flat constructible $\Int_{\ell}$-sheaf $\sheaf{F}$ on $X$ and on the base change $\algc{X}$ of $X$ to the algebraic closure of $\FF_q$, respectively. Furthermore, let $\Frob_{\FF_q}\in\Gal(\algc{\FF}_q/\FF_q)$ denote the geometric Frobenius. The Grothendieck trace formula
$$
L(\sheaf{F},T)=\prod_{i\in\Int}\det(1-\Frob_{\FF_q} T:\HF_c^i(\algc{X},\sheaf{F}))^{(-1)^{i+1}}
$$
implies that the $L$-function $L(\sheaf{F},T)$  of $\sheaf{F}$ is in fact a rational function.

The following theorem is our analogue of the noncommutative Iwasawa main conjecture in the special situation described above:

\begin{thm}\label{thm:special case NCIMC}\
\begin{enumerate}
\item $\RDer\Sectc(X,\sheaf{M}(G)\tensor_{\Int_\ell}\sheaf{F})$ is a perfect complex of $\Int_{\ell}[[G]]$-modules whose cohomology groups are $S$-torsion. Moreover, the endomorphism $\id-\Frob_{\FF_q}$ of the complex $\RDer\Sectc(\algc{X},\sheaf{M}(G)_S\tensor_{\Int_\ell}\sheaf{F})$ is a quasi-iso\-morphism and hence, it gives rise to an element
    $$\ncL_G(X/\FF_q,\sheaf{F})=[\id-\Frob_{\FF_q}]^{-1}\in \KTh_1(\Int_{\ell}[[G]]_S).$$
\item $d\ncL_G(X/\FF_q,\sheaf{F})=[\RDer\Sectc(X,\sheaf{M}(G)\tensor_{\Int_\ell}\sheaf{F})]^{-1}$.
\item Assume that $\rho$ is a continuous $\Int_{\ell}$-representation of $G$. Then $$\rho(\ncL_G(X/\FF_q,\sheaf{F}))=L(\sheaf{M}(\rho)\tensor_{\Int_{\ell}}\sheaf{F},[\Frob_{\FF_q}]^{-1})$$
    in $Q(\Int_{\ell}[[\Gamma]])^{\times}$.
\end{enumerate}
\end{thm}

Theorem~\ref{thm:special case NCIMC}.(3) implies the following interpolation property of $\ncL_G(X/\FF_q,\sheaf{F})$ with respect to special values of $L$-functions: Let $\epsilon$ denote the cyclotomic character and decompose $\epsilon=\epsilon_f\times\epsilon_{\infty}$ according to the decomposition $\Gal(\FF_q(\zeta_{\ell^{\infty}})/\FF_q)=\Delta\times\Gamma$. Then for every $n\in\Int$, the leading term of the image of $\epsilon_{\infty}^n\rho(\ncL_G(X/\FF_q,\sheaf{F}))$ under the isomorphism
$$
 Q(\Int_{\ell}[[\Gamma]])\mto Q(\Int_{\ell}[[X]]),\qquad [\Frob_{\FF_q}]^{-1}\mapsto X+1
$$
agrees with the leading term of $L(\sheaf{M}(\epsilon_f^{-n}\rho)\tensor_{\Int_{\ell}}\sheaf{F},q^{-n}T)$ at $T=1$.

In Section~\ref{sec:IMC} of this article we will prove a version of the noncommutative Iwasawa main conjecture that is in several aspects more general than Theorem~\ref{thm:special case NCIMC}:

The above theorem is limited to geometrically connected schemes $X$. We overcome this limitation by allowing $G$ to be the covering group of any suitable principal covering of $X$. Moreover, we will only require $G$ to be a virtual pro-$\ell$-group, which is a slightly weaker condition than being an $\ell$-adic Lie group.

The ring $\Int_{\ell}$ will be replaced by more general rings of scalars. A good class of rings to work with is the class of adic $\Int_{\ell}$-algebras, which is also used in \cite{FK:CNCIT}. It contains all finite extensions of $\Int_{\ell}$ and is closed under forming profinite group rings with respect to virtual pro-$\ell$-groups. Furthermore, it is more natural to state and prove the conjecture not only for flat and constructible sheaves, but to extend it to perfect complexes.

Beyond proving the main conjecture, we will study the transformation properties of $\ncL_G(X/\FF_q,\sheaf{F})$ under scalar extensions and changes of the principal covering of $X$. Note that parts $(1)$ and $(2)$ of the above theorem and its generalisation remain true in the case $\ell\mid q$, but our element $\ncL_G(X/\FF_q,\sheaf{F})$ does not satisfy $(3)$.

The $\KTh$-theoretical formulation of the conjecture will be based on Waldhausen's construction of higher $\KTh$-groups \cite{Wal:AlgKTheo} and on the construction of the $1$-type of the associated topological spectrum given by F.~Muro and A.~Tonks \cite{MT:1TWKTS}. The localisation sequence may then be viewed as a consequence of Waldhausen's localisation theorem. In the appendix, we derive an explicit description of its rightmost connecting homomorphism. This allows us to obtain a particularly simple description of the elements of $\KTh$-groups appearing in the conjecture.

In order to use Waldhausen's formalism, we need to introduce suitable Waldhausen categories replacing the usual triangulated categories of perfect complexes of adic sheaves. We also need to define Waldhausen exact functors calculating the classical derived functors such as higher derived images with proper support and derived tensor products. For this, we use the approach developed in \cite{Witte:PhD}. All necessary transformation properties for elements in $\KTh$-groups can then be read off directly from the underlying transformation properties of the Waldhausen exact functors.

After these preparations, the proof of part $(1)$ of the above theorem can be reduced to the case that $G\isomorph H\times \Int_{\ell}$ with $H$ finite. In this situation the Hochschild-Serre spectral sequence implies that
$$
H_c^i(X,\sheaf{M}(G)\tensor_{\Int_{\ell}}\sheaf{F})\isomorph \varprojlim_{n} H_c^{i-1}(\algc{Y},\sheaf{F})/(1-\Frob_{\FF_q}^{\ell^n}) H_c^{i-1}(\algc{Y},\sheaf{F}),
$$
where $Y$ is the Galois covering of $X$ corresponding to $H$. Since $H_c^{i-1}(\algc{Y},\sheaf{F})$ is finitely generated as $\Int_{\ell}$-module, it follows that $H_c^i(X,\sheaf{M}(G)\tensor_{\Int_{\ell}}\sheaf{F})$ is $S$-torsion.
Part $(2)$ is a formal consequence of part $(1)$ and our description of the boundary homomorphism in the localisation sequence. Finally, part $(3)$ may be reduced to the classical Grothendieck trace formula and the fact that the evaluation homomorphism
$$
\Int_{\ell}[T]\mto \Int_{\ell}[[\Gamma]],\qquad T\mapsto[\Frob_{\FF_q}]^{-1},
$$
maps $1+T\Int_{\ell}[T]$ to $S$.

The article is structured as follows. In Section~\ref{sec:principal coverings} we recall the necessary terminology of principal coverings. Section~\ref{sec:adic rings} contains a brief account on adic rings. Furthermore, we give a convenient construction of a Waldhausen category calculating their $\KTh$-theory. In Section~\ref{sec:localisation} we introduce the Waldhausen categories used to calculate $\KTh_1(\Int_{\ell}[[G]]_S)$ and $\KTh_0(\Int_{\ell}[[G]],\Int_{\ell}[[G]]_S)$ in the localisation sequence. Our construction of the Waldhausen category of perfect complexes of adic sheaves is recalled in Section~\ref{sec:perf complexes}. In Section~\ref{sec:adic sheaves induced by coverings} we define an analogue of $\sheaf{M}(G)$ for arbitrary principal coverings and study its transformation properties. Section~\ref{sec:the cyclotomic covering} treats the special case $G=\Gamma$. The precise formulations and the proofs of our main results are given in Section~\ref{sec:IMC}. In the appendix we derive an explicit description of the rightmost connecting homomorphism of Waldhausen's localisation sequence under the same assumptions under which the localisation sequence is known to exist.

Using the results of this article and of \cite{EmertonKisin}, D.~Burns \cite{Burns:MCinGIwTh+RelConj} has recently constructed in the case $\ell\mid q$ a modification of $\ncL_G(X/\FF_q,\sheaf{F})$ which has the right interpolation property. The results of this article together with the descent formalism developed by Venjakob and Burns in \cite{BV:DescentTheory} can also be used to generalise the proof of an analogue of the equivariant Tamagawa number conjecture given in \cite{Bu:VEZFCFF}. An analogue of the noncommutative Iwasawa main conjecture for elliptic curves over function fields in the case that $\ell$ is equal to the characteristic~$p$ of the field in question has been considered in \cite{OchiaiTrihan:OnTheSelmer}. F.~Trihan and D.~Vauclair have announced proofs of more general main conjectures in this case. We also point out that tremendous progress towards a proof of the noncommutative Iwasawa main conjecture for totally real fields has been achieved in \cite{Kato:HeisenbergType}, \cite{Hara}, \cite{Kakde1}, \cite{Kakde2}, and in \cite{RW:MainConjecture}.

The author wants to thank A.~Schmidt and O.~Venjakob for valuable discussions and advice.

\section{Principal Coverings}\label{sec:principal coverings}

In this section, we recall the concept of principal coverings from \cite[Def.~2.8]{SGA1}. More precisely, we shall consider pro-objects over the category of finite principal coverings defined in \emph{loc.\ cit.}

If $A$ is either a commutative ring or a scheme, we let
$\cat{Sch}_A$ denote the category of schemes of finite type over
$A$. If $G$ is any profinite group, we
write $\NS_G$ for the set of open normal subgroups of
$G$, partially ordered by inclusion.

\begin{defn}
Let $G$ be a profinite group and $X$ a locally noetherian scheme.
A \emph{principal covering} $(f\colon Y\mto X, G)$ of $X$ with
\emph{Galois group} $G$ is an inverse system of $X$-schemes
$$
(f_U\colon
Y_U\mto X)_{U\in\NS_G},
$$
together with a right operation of $G$ on the system such that for any
$U\in\NS_G$,
\begin{enumerate}
\item $f_U$ is finite, \'etale, and surjective,

\item the operation of $U$ on the scheme $Y_U$ is trivial,

\item the natural morphism
$$
\bigsqcup_{\sigma\in G/U}Y_U\xrightarrow{\bigsqcup \id_{Y_U}\times
\sigma} Y_U\times_X Y_U
$$
is an isomorphism.
\end{enumerate}
\end{defn}

For any profinite group $G$ and a locally noetherian scheme $X$,
there is always the trivial principal covering $(X\times G\mto
X,G)$ given by
$$
(X\times G)_U=\bigsqcup_{\sigma\in G/U} X
$$
for any open normal subgroup $U$ of $G$. If $X$ is connected and
$x$ is a geometric point of $X$, then there exists a distinguished principal
covering $(f\colon \widetilde{X}\mto X, \fundG(X,x))$ whose Galois group is
the \'etale fundamental group $\fundG(X,x)$ of $X$. It is characterised by the property that any compatible system of geometric points $(x\mto \widetilde{X}_U)_{U\in\NS_{\fundG(X,x)}}$ over the base point induces an isomorphism
$$
\varinjlim_{U\in \NS_{\fundG(X,x)}}\Hom_X(\widetilde{X}_U,\cdot)\mto \Hom_X(x,\cdot)
$$
of functors from the category of finite \'etale $X$-scheme to sets. Moreover, the schemes $\widetilde{X}_U$ are connected.

If $(f\colon Y\mto X,G)$ is a principal covering and $X'\mto X$
is a locally noetherian $X$-scheme, then we write
$$
(f\times_X X'\colon Y\times_X X'\mto X',G)
$$
for the principal covering
of $X'$ obtained by base change, i.\,e.\ $(Y\times_X X')_U=Y_U\times_X X'$.

If $V$ is an open (not necessarily normal) subgroup of $G$ and $U\subset V$
is an open normal subgroup of $G$ then the quotient
scheme
$$
Y_V:=Y_U/(V/U)
$$
exists and is (up to canonical isomorphism) independent of the choice of $U$. Moreover, we
obtain a principal covering
$$
(f^V\colon Y\mto Y_V,V)
$$
of $Y_V$ given by the inverse system $(f^V_U\colon Y_U\mto Y_V)_{U\in\NS_V}$, where $f^V_U$ denotes the canonical projection map.

If $H$ is a closed normal subgroup of $G$ and $\alpha\colon G\mto G/H$ the natural projection, we define
$$
(f_H\colon Y_H\mto X,G/H)
$$
to be the principal covering given by the inverse system
$$
(f_{\alpha^{-1}(U)}\colon Y_{\alpha^{-1}(U)}\mto X)_{U\in \NS_{G/H}}.
$$

\begin{defn}
A \emph{morphism}
$$
a\colon\quad(Y\mto X,G)\quad\mto\quad (Y'\mto X,G')
$$
of principal coverings of $X$ is a continuous group homomorphism $\alpha\colon G\mto G'$ together
with a $G$-equivariant morphism of inverse systems
$$
a\colon\quad(Y_{\alpha^{-1}(U)}\mto X)_{U\in\NS_{G'}}\quad\mto\quad(Y'_U\mto X)_{U\in\NS_{G'}}.
$$
\end{defn}

\begin{lem}\label{lem:ismorphisms of G-coverings}
A morphism
$$
a\colon\quad(Y\mto X,G)\quad\mto\quad (Y'\mto X,G')
$$
of principal coverings of $X$ is an
isomorphism if and only if the associated homomorphism of groups
$\alpha\colon G\mto G'$ is an isomorphism.
\end{lem}
\begin{proof}
We may assume that $G=G'$ and that $\alpha$ is the identity.
We may then reduce to the case that $G$ is finite and that $X$
is the spectrum of a local ring $A$. Then $Y$ and $Y'$ are the spectra of finite flat $A$-algebras $B$ and $B'$, respectively. The rank of both $B$ and $B'$ as free $A$-modules is equal to the cardinality of $G$. Since $a\colon Y\mto Y'$ is finite \'etale, it follows that $B$ is a finitely generated, projective $B'$-module of constant rank $1$. Hence, $B\isomorph B'$.
\end{proof}

In the following, we will impose further restrictions on the group $G$.

\begin{defn}\label{defn:l-admissible}
Let $\ell$ be a prime. We call a profinite group $G$ a
\emph{virtual pro-$\ell$-group} if its $\ell$-Sylow subgroups are of finite index. Without further comment, we require all virtual pro-$\ell$-groups appearing in this article to be topologically finitely generated.

A principal covering is called a \emph{virtual pro-$\ell$-covering} if its
Galois group is a virtual pro-$\ell$-group.
\end{defn}

Note that all compact $\ell$-adic Lie groups are virtual pro-$\ell$-groups in the above sense. The following example of a virtual pro-$\ell$-covering will play an important role: Let $\FF_q$ be a finite field with $q$ elements. We fix an algebraic closure $\algc{\FF}_q$ of $\FF_q$.  Let $\ell$ be any prime, let $k$ be an integer prime to $\ell$, and set
$$
\FF_{q^{k\ell^{\infty}}}=\bigcup_{n\geq 0} \FF_{q^{k\ell^n}}
$$
(as subfield of $\algc{\FF}_q$). This gives rise to a principal covering
$$
(\Spec \FF_{q^{k\ell^{\infty}}}\mto
\Spec \FF_q,\Gamma_{k\ell^{\infty}})
$$
with Galois group $\Gamma_{k\ell^{\infty}}\isomorph\Int/k\Int\times\Int_{\ell}$.

\begin{defn}
Let $X$ be a scheme of finite type over the finite field $\FF_q$ and set $X_{k\ell^{\infty}}=X\times_{\Spec \FF_q}\Spec \FF_{q^{k\ell^{\infty}}}$. The
principal covering
$$
(X_{k\ell^{\infty}} \mto X,\Gamma_{k\ell^{\infty}})
$$
will be called the \emph{cyclotomic $\Gamma_{k\ell^{\infty}}$-covering} of
$X$.
\end{defn}

We point out that with this definition, $X_{k\ell^{\infty}}$ is not necessarily connected, even if $X$ itself is connected.

\begin{defn}
Let $X$ be a scheme of finite type over a finite field $\FF$, $\ell$ an arbitrary prime number. We call a
principal covering $(f\colon Y\mto X,G)$ \emph{admissible} if
\begin{enumerate}
\item $G\isomorph H\rtimes\Gamma_{\ell^{\infty}}$ is the semidirect product of a closed
normal virtual pro-$\ell$-subgroup $H$ and the group
$\Gamma_{\ell^{\infty}}$,

\item $(f_H\colon Y_H\mto X,\Gamma_{\ell^{\infty}})$ is isomorphic to the cyclotomic
$\Gamma_{\ell^{\infty}}$-covering of $X$.
\end{enumerate}
\end{defn}

Note that the semidirect product $H\rtimes\Gamma_{\ell^{\infty}}$ of a virtual
pro-$\ell$-group $H$ and $\Gamma_{\ell^{\infty}}\isomorph\Int_{\ell}$ is
itself a virtual pro-$\ell$-group.

\section{The \texorpdfstring{$\KTh$}{K}-Theory of Adic Rings}\label{sec:adic rings}

In this section, we recall some facts about adic rings and their $\KTh$-theory. Properties of these rings have previously been studied in \cite{Warner:TopRings} and in \cite{FK:CNCIT}. We refer to \cite[Section~5.1-2]{Witte:PhD} for a more complete treatment.

All rings will be associative with unity, but not necessarily
commutative. For any ring $R$, we let
$$
\Jac(R)=\set{x\in R\;|\;\text{$1-rx$ is invertible for any $r\in R$}}
$$
denote the \emph{Jacobson radical} of $R$. The ring $R$
is called \emph{semilocal} if $R/\Jac(R)$ is artinian.

\begin{defn}
A ring $\Lambda$ is called an \emph{adic ring} if for each integer
$n\geq 1$, the ideal $\Jac(\Lambda)^n$ is of finite index in
$\Lambda$ and
$$
\Lambda=\varprojlim_{n}\Lambda/\Jac(\Lambda)^n.
$$
\end{defn}

Note that $\Lambda$ is adic precisely if it is compact, semilocal
and the Jacobson radical is finitely generated \cite[Theorem~36.39]{Warner:TopRings}.

For any adic ring $\Lambda$ we denote by $\openideals_{\Lambda}$ the set of open two-sided ideals of $\Lambda$, partially ordered by inclusion.

\begin{prop}\label{prop:l-admissible group rings are adic}
Let $\Lambda$ be an adic $\Int_{\ell}$-algebra and let $G$ be a virtual pro-$\ell$-group. Then the profinite group ring
$$
\Lambda[[G]]=\varprojlim_{J\in\openideals_{\Lambda}}
\varprojlim_{U\in\NS_{G}}\Lambda/J[G/U]
$$
is an adic $\Int_{\ell}$-algebra. Moreover, if $U$ is any open
normal pro-$\ell$-subgroup of $G$, then the kernel of
$$
\Lambda[[G]]\mto \Lambda/\Jac(\Lambda)[G/U]
$$
is contained in $\Jac(\Lambda[[G]])$.
\end{prop}
\begin{proof}
We begin by proving the assertion about the Jacobson radical. Clearly, $\Lambda[[G]]$ is a compact ring. Hence,
$$
\Jac(\Lambda[[G]])=\varprojlim_{V\in\NS_{G},n\geq 0} \Jac(\Lambda/\Jac(\Lambda)^n[G/V]).
$$
We may thus assume that
$\Lambda$ and $G$ are finite. A direct calculation shows that the
twosided ideal $\Jac(\Lambda)\Lambda[G]$ is nilpotent and therefore, it is contained in the Jacobson radical of $\Lambda[G]$. Consequently, we
may assume that $n=1$, i.\,e.\ $\Lambda$ is a finite product of
full matrix rings over finite fields of characteristic $\ell$.
Considering each factor of $\Lambda$ separately and using that
$$
\Jac(M_{k,k}(\Lambda)[G])=\Jac(M_{k,k}(\Lambda[G]))=M_{k,k}(\Jac(\Lambda[G])),
$$
we can restrict to $\Lambda$ itself being a finite field of
characteristic $\ell$. We are thus reduced to the classical case
treated in \cite[Prop.~5.26]{CurtisReinerI}.

Hence, returning to the general situation, we find an open normal pro-$\ell$
subgroup $U$ of $G$ such that the kernel of
$$
\Lambda[[G]]\mto \Lambda/\Jac(\Lambda)[G/U]
$$
is contained in $\Jac(\Lambda[[G]])$. This kernel is an open ideal of $\Lambda[[G]]$ generated by a system of generators of $\Jac(\Lambda)$ over $\Lambda$ together with the elements $1-u_i$ for a system of topological generators $(u_i)$ of $U$. Thus, $\Jac(\Lambda[[G]])$  is also open and finitely generated. Therefore, we conclude that $\Lambda[[G]]$ is an adic ring.
\end{proof}

We will now examine the algebraic $\KTh$-groups of $\Lambda$. For
this, we will follow Waldhausen's approach \cite{Wal:AlgKTheo}.
Recall that a \emph{Waldhausen category} is
a category $\cat{W}$ with zero object together with two classes of
morphisms, called \emph{cofibrations} and \emph{weak
equivalences}, that satisfy a certain set of axioms. Using
Waldhausen's $S$-construction one can associate to each such
category in a functorial manner a connected pointed topological space
$X(\cat{W})$. By definition, the $n$-th $K$-group of $\cat{W}$ is
the $(n+1)$-th homotopy group of this space:
$$
\KTh_n(\cat{W})=\pi_{n+1}(X(\cat{W})).
$$
\emph{Waldhausen exact functors} are functors that respect the
additional structure of a Waldhausen category. Each such functor
$F\colon \cat{W}\mto \cat{W}'$ induces a continuous map between
the associated topological spaces and hence, a homomorphism
$$
\KTh_n(F)\colon \KTh_n(\cat{W})\mto\KTh_n(\cat{W}').
$$
We refer to \cite{ThTr:HAKTS+DC} for a more thorough introduction to the
topic.

Let $R$ be any ring. Recall that a complex $\cmplx{M}$ of left $R$-modules is
called \emph{strictly bounded} if there exists a number $k$ such that $M^n=0$ for $n<-k$ and for $n>k$. The complex $\cmplx{M}$ is called \emph{strictly perfect} if it is strictly bounded and for
every $n$, the module $M^n$ is finitely generated and projective.
The complex $\cmplx{M}$ is called \emph{perfect} if it
quasi-isomorphic to a strictly perfect complex in the category of
all complexes of left $R$-modules.

\begin{defn}
We let $\cat{SP}(R)$ denote the
Waldhausen category of strictly perfect complexes and $\cat{P}(R)$
the Waldhausen category of perfect complexes: The morphisms are morphisms of complexes $f\colon \cmplx{M}\rightarrow \cmplx{N}$ in the usual sense (not morphisms in the derived category). The weak equivalences are the
quasi-isomorphisms and the cofibrations are
morphisms of complexes $f\colon \cmplx{M}\rightarrow \cmplx{N}$  which are injective (i.\;e.\ for each $n$, $f\colon M^n\rightarrow N^n$ is injective).
\end{defn}

By the Gillet-Waldhausen Theorem \cite[Theorem 1.11.7]{ThTr:HAKTS+DC} we know that the Waldhausen
$\KTh$-theory of $\cat{SP}(R)$ and of $\cat{P}(R)$ coincide with the Quillen
$\KTh$-theory of $R$:
$$\KTh_n(\cat{P}(R))=\KTh_n(\cat{SP}(R))=\KTh_n(R).$$

\begin{defn}\label{defn:complexes of bimodules}
Let $R$ and $S$ be rings. We denote by
$R^{\op}\text{-}\cat{SP}(S)$ the Waldhausen category of complexes
of $S$-$R$-bimodules (with $S$ acting from the left, $R$ acting
from the right) which are strictly perfect as complexes of
$S$-modules. The weak equivalences and cofibrations are the same as in $\cat{SP}(S)$.
\end{defn}

For complexes $\cmplx{M}$ and $\cmplx{N}$ of right and left
$R$-modules, respectively, we let
$$
\cmplx{(M\tensor_R N)}
$$
denote
the total complex of the bicomplex $\cmplx{M}\tensor_R \cmplx{N}$.
Any complex $\cmplx{M}$ in $R^{\op}\text{-}\cat{SP}(S)$
gives rise to a Waldhausen exact functor
$$
\cmplx{(M\tensor_R(-))}\colon \cat{SP}(R)\mto \cat{SP}(S).
$$
and hence, to homomorphisms $ \KTh_n(R)\mto \KTh_n (S)$.

Let now $\Lambda$ be an adic ring.
We introduce another Waldhausen category
computing the $\KTh$-theory of $\Lambda$ which is more convenient for our purposes.

\begin{defn}
Let $R$ be any ring. A complex $\cmplx{M}$ of left $R$-modules is
called \emph{$DG$-flat} if every module $M^n$ is flat and for
every acyclic complex $\cmplx{N}$ of right $R$-modules, the
complex $\cmplx{(N\tensor_R M)}$ is acyclic.
\end{defn}

Note that a strictly bounded above complex $\cmplx{M}$ is $DG$-flat precisely if every module $M^n$ is flat. The notion of $DG$-flatness, which was introduced in \cite{AvrFoxby:HomDimUnbCompl}, allows us to deal also with unbounded complexes. These are often the most natural objects to consider in our context. More precisely, we will have to deal with inverse systems of such objects:

\begin{defn}\label{defn:PDG(Lambda)}
Let $\Lambda$ be an adic ring. We denote by
$\cat{PDG}^{\cont}(\Lambda)$ the following Waldhausen category.
The objects of $\cat{PDG}^{\cont}(\Lambda)$ are inverse system
$(\cmplx{P}_I)_{I\in \openideals_{\Lambda}}$ satisfying the
following conditions:
\begin{enumerate}
\item for each $I\in\openideals_{\Lambda}$, $\cmplx{P}_I$ is a
$DG$-flat perfect complex of left $\Lambda/I$-modules,

\item for each $I\subset J\in\openideals_{\Lambda}$, the
transition morphism of the system
$$
\varphi_{IJ}:\cmplx{P}_I\mto \cmplx{P}_J
$$
induces an isomorphism
$$
\Lambda/J\tensor_{\Lambda/I}\cmplx{P}_I\isomorph \cmplx{P}_J.
$$
\end{enumerate}
A morphism of inverse systems $(f_I\colon \cmplx{P}_I\mto
\cmplx{Q}_I)_{I\in\openideals_{\Lambda}}$ in
$\cat{PDG}^{\cont}(\Lambda)$ is a weak equivalence if every $f_I$
is a quasi-isomorphism. It is a cofibration if every $f_I$ is
injective.
\end{defn}

\begin{prop}\label{prop:embedding SP(Lambda) in PDG(Lambda)}
The Waldhausen exact functor
$$
F\colon \cat{SP}(\Lambda)\mto\cat{PDG}^{\cont}(\Lambda),\qquad
\cmplx{P}\mto
(\Lambda/I\tensor_{\Lambda}\cmplx{P})_{I\in\openideals_{\Lambda}}
$$
identifies $\cat{SP}(\Lambda)$ with a full Waldhausen subcategory
of $\cat{PDG}^{\cont}(\Lambda)$ such that for every $\cmplx{Q}$ in $\cat{PDG}^{\cont}(\Lambda)$ there exists a complex $\cmplx{P}$ in $\cat{SP}(\Lambda)$ and a quasi-isomorphism $F(\cmplx{P})\wto \cmplx{Q}$. Moreover, $F$ induces isomorphisms
$$
\KTh_n(\cat{SP}(\Lambda))\isomorph\KTh_n(\cat{PDG}^{\cont}(\Lambda)).
$$
\end{prop}
\begin{proof}
The main step is to show that for every object $(\cmplx{Q}_I)_{I\in\openideals_{\Lambda}}$ in $\cat{PDG}^{\cont}(\Lambda)$, the complex
$$
\varprojlim_{I\in\openideals_{\Lambda}}\cmplx{Q}_I
$$
is a perfect complex of $\Lambda$-modules. This is proved using the argument of \cite[Proposition~1.6.5]{FK:CNCIT}. The assertion about the $\KTh$-theory is then an easy consequence of the Waldhausen approximation theorem.
We refer to \cite[Proposition~5.2.5]{Witte:PhD} for the details.
\end{proof}

\begin{rem}
Definition~\ref{defn:PDG(Lambda)} makes sense for any compact ring $\Lambda$. However, we do not expect Proposition~\ref{prop:embedding SP(Lambda) in PDG(Lambda)} to be true in this generality. The argument of  \cite[Proposition~1.6.5]{FK:CNCIT} uses in an essential way that $\Lambda$ is compact for its $\Jac(\Lambda)$-adic topology.
\end{rem}

We can extend the definition of the tensor product to
$\cat{PDG}^{\cont}(\Lambda)$ as follows.

\begin{defn}\label{defn:change of ring functor}
For $(\cmplx{P}_I)_{I\in\openideals_{\Lambda}}\in
\cat{PDG}^{\cont}(\Lambda)$ and $\cmplx{M}\in
\Lambda^{\op}\text{-}\cat{SP}(\Lambda')$
we define a Waldhausen exact functor
$$
\ringtransf_{\cmplx{M}}\colon
\cat{PDG}^{\cont}(\Lambda)\mto\cat{PDG}^{\cont}(\Lambda'),\qquad \cmplx{P}\mto(\varprojlim_{J\in\openideals_{\Lambda}}
\Lambda'/I\tensor_{\Lambda'}\cmplx{(M\tensor_{\Lambda}P_{J})})_{I\in\openideals_{\Lambda'}}
$$
\end{defn}

Note that for every $I\in\openideals_{\Lambda'}$ there exists a $J_0\in \openideals_{\Lambda}$ such that
$$
\varprojlim_{J\in\openideals_{\Lambda}}
\Lambda'/I\tensor_{\Lambda'}\cmplx{(M\tensor_{\Lambda}P_{J})}=\cmplx{(M/IM\tensor_{\Lambda/J_0}P_{J_0})}.
$$
One checks easily that this definition is compatible with the usual tensor product with $\cmplx{M}$ on $\cat{SP}(\Lambda)$.

From \cite{MT:1TWKTS} we deduce the following generators and relations for the
group $\KTh_1(\Lambda)$.

\begin{prop}\label{prop:presentation of K_1}
The group $\KTh_1(\Lambda)$ is  generated by quasi-isomorphisms
$$(f_I\colon \cmplx{P}_I\wto
\cmplx{P}_I)_{I\in\openideals_{\Lambda}}
$$
in $\cat{PDG}^{\cont}(\Lambda)$. Moreover, the following
relations are satisfied:
\begin{enumerate}
\item $[(f_I\colon\cmplx{P}_I\wto
\cmplx{P}_I)_{I\in\openideals_{\Lambda}}]=[(g_I\colon\cmplx{P}_I\wto
\cmplx{P}_I)_{I\in\openideals_{\Lambda}}][(h_I\colon\cmplx{P}_I\wto
\cmplx{P}_I)_{I\in\openideals_{\Lambda}}]$ if for each
$I\in\openideals_{\Lambda}$, one has $f_I=g_I\comp h_I$,

\item $[(f_I\colon\cmplx{P}_I\wto
\cmplx{P}_I)_{I\in\openideals_{\Lambda}}]=[(g_I\colon\cmplx{Q}_I\wto
\cmplx{Q}_I)_{I\in\openideals_{\Lambda}}]$ if for each
$I\in\openideals_{\Lambda}$, there exists a quasi-isomorphism
$a_I\colon \cmplx{P}_I\wto \cmplx{Q}_I$ such that the square
$$
\xymatrix{\cmplx{P}_I\ar[r]^{f_I}\ar[d]^{a_I}&\cmplx{P}_I\ar[d]^{a_I}\\
\cmplx{Q}_I\ar[r]^{g_I}&\cmplx{Q}_I}
$$
commutes up to homotopy,

\item $[(g_I\colon\cmplx{P'}_I\wto
\cmplx{P'}_I)_{I\in\openideals_{\Lambda}}]=[(f_I\colon\cmplx{P}_I\wto
\cmplx{P}_I)_{I\in\openideals_{\Lambda}}][(h_I\colon\cmplx{P''}_I\wto
\cmplx{P''}_I)_{I\in\openideals_{\Lambda}}]$ if for each
$I\in\openideals_{\Lambda}$, there exists an exact sequence $\cmplx{P}\cto
\cmplx{P'}\qto \cmplx{P''}$ such that the diagram
$$
\xymatrix{\cmplx{P}_I\ar[d]^{f_I}\ar@{>->}[r]&\cmplx{P'}_I\ar[d]^{g_I}\ar@{>>}[r]&\cmplx{P''}\ar[d]^{h_I}\\
\cmplx{P}_I\ar@{>->}[r]&\cmplx{P'}_I\ar@{>>}[r]&\cmplx{P''}}
$$
commutes in the strict sense.
\end{enumerate}
\end{prop}
\begin{proof}
The description of $\KTh_1(\cat{PDG}^{\cont}(\Lambda))$ as the
kernel of
$$
\mathcal{D}_1\cat{PDG}^{\cont}(\Lambda)\xrightarrow{\del}\mathcal{D}_0\cat{PDG}^{\cont}(\Lambda)
$$
given in \cite[Def.~1.4]{MT:1TWKTS} (see also \ref{defn:1-type}) shows that all endomorphisms which are quasi-isomorphisms do indeed give rise to elements of $\KTh_1(\cat{PDG}^{\cont}(\Lambda))$. Together
with the isomorphism
$$
\KTh_1(\Lambda)\isomorph \varprojlim_{I\in\openideals_{\Lambda}}\KTh_1(\Lambda/I)
$$
\cite[Prop. 1.5.3]{FK:CNCIT} this
description also implies that relations (1) and (3) are satisfied.
For relation (2), one can use \cite[Lemma 3.1.6]{Witte:PhD}.
Finally, the classical description of $\KTh_1(\Lambda)$ implies
that $\KTh_1(\cat{PDG}^{\cont}(\Lambda))$ is already generated by
isomorphisms of finitely generated, projective modules viewed as
strictly perfect complexes concentrated in degree $0$.
\end{proof}

\section{Localisation}\label{sec:localisation}

Localisation is considered a difficult topic in noncommutative
ring theory. To be able to localise at a set of elements $S$ in a
noncommutative ring $R$ one needs to show that this set is a
denominator set. In particular, one must verify the Ore condition
which is often a tedious task. We can avoid this topic
by localising the associated Waldhausen category of perfect
complexes instead of the ring itself.

Let $\Lambda$ be an adic $\Int_{\ell}$-algebra, $H$ a closed subgroup of a profinite group $G$ and assume that both $G$ and $H$ are virtual pro-$\ell$-groups. We define the following Waldhausen categories.

\begin{defn}
We write $\cat{PDG}^{\cont,w_H}(\Lambda[[G]])$ for the full
Waldhausen subcategory of $\cat{PDG}^{\cont}(\Lambda[[G]])$ of
objects $(\cmplx{P}_J)_{J\in\openideals_{\Lambda[[G]]}}$ such that
$$
\varprojlim_{J\in\openideals_{\Lambda[[G]]}} \cmplx{P}_J
$$
is a perfect complex of $\Lambda[[H]]$-modules.

We write $w_H\cat{PDG}^{\cont}(\Lambda[[G]])$ for the Waldhausen
category with the same objects, morphisms and cofibrations as
$\cat{PDG}^{\cont}(\Lambda[[G]])$, but with a new set of weak
equivalences given by those morphisms whose cones are objects of the category
$\cat{PDG}^{\cont,w_H}(\Lambda[[G]])$.
\end{defn}

Note that $\cat{PDG}^{\cont,w_H}(\Lambda[[G]])$ is a full additive
subcategory of $\cat{PDG}^{\cont}(\Lambda[[G]])$ and that it is
closed under weak equivalences, shifts, and extensions. This
implies immediately that both
$\cat{PDG}^{\cont,w_H}(\Lambda[[G]])$ and
$w_H\cat{PDG}^{\cont}(\Lambda[[G]])$ are indeed Waldhausen
categories (see e.\,g.\ \cite[Section~3]{HirMochi:NegativeK}) and
that the natural functors
$$
\cat{PDG}^{\cont,w_H}(\Lambda[[G]])\mto\cat{PDG}^{\cont}(\Lambda[[G]])\mto
w_H\cat{PDG}^{\cont}(\Lambda[[G]])
$$
induce a cofibre sequence of the associated $\KTh$-theory spaces
and hence, a long exact localisation sequence
\begin{align*}
\cdots\mto\KTh_i(\cat{PDG}^{\cont,w_H}(\Lambda[[G]]))\mto\KTh_i(\cat{PDG}^{\cont}(\Lambda[[G]]))\mto\\
\KTh_i(w_H\cat{PDG}^{\cont}(\Lambda[[G]]))\mto\KTh_{i-1}(\cat{PDG}^{\cont,w_H}(\Lambda[[G]]))\mto\cdots
\end{align*}
\cite[Theorem~1.8.2]{ThTr:HAKTS+DC}.

Assume for the moment that $\Lambda=\Int_{\ell}$ and that
$G\isomorph H\rtimes\Gamma_{\ell^{\infty}}$ with $H$ a compact $\ell$-adic Lie group and
$\Gamma_{\ell^{\infty}}\isomorph \Int_{\ell}$. Then Venjakob constructed a left and right
Ore set $S$ of non-zero divisors in $\Int_{\ell}[[G]]$. In particular,
the quotient ring $S^{-1}\Int_{\ell}[[G]]$ exists and is flat as
right $\Int_{\ell}[[G]]$-module. Moreover, a finitely generated
$\Int_{\ell}[[G]]$-module is $S$-torsion if and only if it is
finitely generated as $\Int_{\ell}[[H]]$-module
\cite[Section~2]{CFKSV}.

Since any choice of a generator of $\Gamma_{\ell^{\infty}}$ identifies $\Int_{\ell}[[G]]$ with a skew power series ring over the
noetherian ring $\Int_{\ell}[[H]]$, we see that
$$
\Int_{\ell}[[G]]\isomorph\prod_{n\geq 0} \Int_{\ell}[[H]]
$$
is flat as $\Int_{\ell}[[H]]$-module. In particular, a complex
$(\cmplx{P}_J)_{J\in\openideals_{\Int_{\ell}[[G]]}}$ in
$\cat{PDG}^{\cont}(\Int_{\ell}[[G]])$ is in
$\cat{PDG}^{\cont,w_H}(\Int_{\ell}[[G]])$ if and only if
$$
S^{-1}\Int_{\ell}[[G]]\tensor_{\Int_{\ell}[[G]]}\varprojlim_{J\in\openideals_{\Int_{\ell}[[G]]}}\cmplx{P}_J
$$
is acyclic.

From the localisation theorem in \cite{WeibYao:Localization} we
conclude that in this case,
$$
\KTh_i(\cat{PDG}^{\cont,w_H}(\Int_{\ell}[[G]]))=\KTh_i(\Int_{\ell}[[G]],S^{-1}\Int_{\ell}[[G]])
$$
is the relative $\KTh$-group and that the functor
\begin{align*}
w_H\cat{PDG}^{\cont}(\Int_{\ell}[[G]])&\mto\cat{P}(S^{-1}\Int_{\ell}[[G]]),\\
(\cmplx{P}_J)_{J\in\openideals_{\Int_{\ell}[[G]]}}&\mapsto
S^{-1}\Int_{\ell}[[G]]\tensor_{\Int_{\ell}[[G]]}\varprojlim_{J\in\openideals_{\Int_{\ell}[[G]]}}\cmplx{P}_J
\end{align*}
induces isomorphisms
$$
\KTh_i(w_H\cat{PDG}^{\cont}(\Int_{\ell}[[G]]))=
\begin{cases}
\KTh_i(S^{-1}\Int_{\ell}[[G]]) & \text{if $i>0$,}\\
\im \KTh_0(\Int_{\ell}[[G]])\mto \KTh_0(S^{-1}\Int_{\ell}[[G]]) &
\text{if $i=0$}
\end{cases}
$$
(see also \cite[Prop.~5.3.4]{Witte:PhD}).

We need this more explicit description of $\KTh_1(w_H\cat{PDG}^{\cont}(\Lambda[[G]]))$ only in the following situation.

\begin{lem}\label{lem:K_1 in the commutative case}
Let $\Lambda$ be a commutative adic $\Int_{\ell}$-algebra, $H=1$ and $G=\Gamma_{k\ell^{\infty}}$ with $k$ prime to $\ell$. Set
$$
S=\{f\in\Lambda[[\Gamma_{k\ell^{\infty}}]]\colon \text{$[f]\in\Lambda/\Jac(\Lambda)[[\Gamma_{k\ell^{\infty}}]]$ is a nonzerodivisor}\}
$$
Then
$$
\KTh_1(w_1\cat{PDG}^{\cont}(\Lambda[[\Gamma_{k\ell^{\infty}}]]))=\KTh_1(\Lambda[[\Gamma_{k\ell^{\infty}}]]_S)
=\Lambda[[\Gamma_{k\ell^{\infty}}]]_S^{\times}.
$$
\end{lem}
\begin{proof}
We will show that a strictly perfect complex $\cmplx{P}$ of $\Lambda[[\Gamma_{k\ell^{\infty}}]]$-modules is perfect as complex of $\Lambda$-modules if and only if its cohomology groups are $S$-torsion. Then  the localisation theorem in \cite{WeibYao:Localization} implies that
$$
\KTh_n(w_1\cat{PDG}^{\cont}(\Lambda[[\Gamma_{k\ell^{\infty}}]]))=\KTh_n(\Lambda[[\Gamma_{k\ell^{\infty}}]]_S)
$$
for $n\geq 1$. Since $\Lambda[[\Gamma_{k\ell^{\infty}}]]_S$ is clearly a commutative semilocal ring, the determinant map induces an isomorphism
$$
\KTh_1(\Lambda[[\Gamma_{k\ell^{\infty}}]]_S)
\isomorph\Lambda[[\Gamma_{k\ell^{\infty}}]]_S^{\times}
$$
\cite[Theorem~40.31]{CurtisReinerII}.

Recall that a commutative adic ring is always noetherian and that a bounded complex over a noetherian ring is perfect if and only if its cohomology modules are finitely generated. Hence, it suffices to show that a finitely generated $\Lambda[[\Gamma_{\ell^{k\infty}}]]$-module $M$ is finitely generated as $\Lambda$-module if and only if it is $S$-torsion. This is in turn a direct consequence of the structure theory for finitely generated modulues over the classical Iwasawa-algebra.
\end{proof}

For general $G$, $H$ and $\Lambda$, it is not difficult to prove
that the first $\KTh$-group of
$w_H\cat{PDG}^{\cont}(\Lambda[[G]])$ agrees with the corresponding
localised $\KTh_1$-group defined in \cite[Def.~1.3.2]{FK:CNCIT}:
$$
\KTh_1(w_H\cat{PDG}^{\cont}(\Lambda[[G]]))=
\KTh_1(\cat{PDG}^{\cont}(\Lambda[[G]]),\cat{PDG}^{\cont,w_H}(\Lambda[[G]])),
$$
but we will make no use of this.

Next, let $\Lambda$ and $\Lambda'$ be two adic $\Int_{\ell}$-algebras and $G$, $G'$, $H$, $H'$ be virtual pro-$\ell$-groups. Assume that $H$ and $H'$ are closed subgroups of $G$ and $G'$, respectively. We want to investigate under which circumstances the Waldhausen exact functor
$$
\ringtransf_{\cmplx{K}}\colon \cat{PDG}^{\cont}(\Lambda[[G]])\mto \cat{PDG}^{\cont}(\Lambda'[[G']])
$$
for an object $\cmplx{K}$ in $\Lambda[[G]]^{\op}-\cat{SP}(\Lambda'[[G']])$ restricts to a functor
$$
\cat{PDG}^{\cont,w_H}(\Lambda[[G]])\mto \cat{PDG}^{\cont,w_{H'}}(\Lambda'[[G']]).
$$
Note that if this is the case, then $\ringtransf_{\cmplx{K}}$ also extends to a functor
$$
w_H\cat{PDG}^{\cont}(\Lambda[[G]])\mto w_{H'}\cat{PDG}^{\cont}(\Lambda'[[G']]).
$$
Both functors will again be denoted by $\ringtransf_{\cmplx{K}}$.

For any
compact ring $\Omega$, we let
$$
M\hat{\tensor}_\Omega N=\varprojlim_{U,V}M/U\tensor_{\Omega} N/V
$$
denote the \emph{completed tensor product} of the compact right
$\Omega$-module $M$ with the compact left $\Omega$-module $N$.
Here, $U$ and $V$ run through the open submodules of $M$ and $N$,
respectively. Note that the completed tensor product $M\hat{\tensor}_{\Omega}N$
agrees with the usual tensor product $M\tensor_{\Omega}N$ if
either $M$ or $N$ is finitely presented.

\begin{defn}
We call a compact $\Omega$-module $P$
\emph{compact-flat} if the completed tensor product with $P$
preserves continuous injections of compact modules.
\end{defn}

If the compact ring $\Omega$ is noetherian, then $P$ is compact-flat precisely if it is flat, but in general, the two notions do not need to coincide.

\begin{lem}\label{lem:compactflatness property}
Let $\Lambda$ be an adic $\Int_\ell$-algebra, $H$ a closed subgroup of $G$ such that both $G$ and $H$ are virtual pro-$\ell$-groups. Then any
finitely generated, projective $\Lambda[[G]]$-module is
compact-flat as $\Lambda[[H]]$-module.
\end{lem}
\begin{proof}
It suffices to prove the lemma for the finitely generated,
projective $\Lambda[[G]]$-module $\Lambda[[G]]$. Then the
statement follows since for every $n$ and every open normal
subgroup $U$ in $G$, $\Lambda/\Jac^n(\Lambda)[G/U]$ is flat as
$\Lambda/\Jac^n(\Lambda)[H/H\cap U]$-module.
\end{proof}

\begin{lem}\label{lem: compactperfect=perfect}
Let $\Lambda$ be an adic ring and $\cmplx{P}$ a strictly bounded
complex of compact-flat left $\Lambda$-modules. Then $\cmplx{P}$
is a perfect complex of $\Lambda$-modules if and only if
$\Lambda/\Jac(\Lambda)\tensor_{\Lambda}\cmplx{P}$ has finite
cohomology groups.
\end{lem}
\begin{proof}
Assume that $\cmplx{Q}\wto\cmplx{P}$ is a quasi-isomorphism with
$\cmplx{Q}$ strictly perfect. Then the quasi-isomorphism
$$
\Lambda/\Jac(\Lambda)\tensor_{\Lambda}\cmplx{Q}=\Lambda/\Jac(\Lambda)\hat{\tensor}_{\Lambda}\cmplx{Q}
\wto \Lambda/\Jac(\Lambda)\hat{\tensor}_{\Lambda}\cmplx{P}
$$
shows that $\Lambda/\Jac(\Lambda)\hat{\tensor}_{\Lambda}\cmplx{P}$
has finite cohomology groups. Since $\Lambda/\Jac(\Lambda)$ is finite and $\Jac(\Lambda)$ is finitely generated, we can replace the completed tensor product by the usual tensor product.

Conversely, assume that
$\Lambda/\Jac(\Lambda)\tensor_{\Lambda}\cmplx{P}$ has finite
cohomology groups. Without loss of generality we may suppose that
$P^k=0$ for $k<0$ and $k>n$ with some $n\geq 0$. By assumption,
$\Lambda/\Jac(\Lambda)\tensor_{\Lambda} \HF^n(P)$ is finite.
By the topological Nakayama lemma we conclude that the compact
module $\HF^n(P)$ is finitely generated. We proceed by induction
over $n$ to prove the perfectness of $\cmplx{P}$. If $n=0$, we see that $P^0$ is
finitely generated. Using the lifting of idempotents in $\Lambda$, we conclude that $P^0$ is also projective. If $n>0$, we may choose a
homomorphism $f\colon\Lambda^k[-n]\mto \cmplx{P}$ such that
$\HF^n(f)$ is surjective. The cone of this morphism then satisfies
the induction hypothesis. Since the category of perfect complexes is closed under
extensions we conclude that $\cmplx{P}$ is perfect.
\end{proof}

We can now state the following criterion:

\begin{prop}\label{prop:localising bimodules}
Let $\Lambda$ and $\Lambda'$ be two adic $\Int_{\ell}$-algebras and $G$, $G'$, $H$, $H'$ be virtual pro-$\ell$-groups. Assume that $H$ and $H'$ are closed subgroups of $G$ and $G'$, respectively. Suppose that $\cmplx{K}$ is a complex in $\Lambda[[G]]^{\op}$-$\cat{SP}(\Lambda'[[G']])$ such that there exists a complex $\cmplx{L}$  in $\Lambda[[H]]^{\op}$-$\cat{SP}(\Lambda'[[H']])$ and a quasi-isomorphism of complexes of $\Lambda'[[H']]$-$\Lambda[[G]]$-bimodules
$$
\cmplx{L}\hat{\tensor}_{\Lambda[[H]]}\Lambda[[G]]\wto \cmplx{K}.
$$
Then
$$
\ringtransf_{\cmplx{K}}\colon \cat{PDG}^{\cont}(\Lambda[[G]])\mto \cat{PDG}^{\cont}(\Lambda'[[G']])
$$
restricts to
$$
\ringtransf_{\cmplx{K}}\colon\cat{PDG}^{\cont,w_H}(\Lambda[[G]])\mto \cat{PDG}^{\cont,w_{H'}}(\Lambda'[[G']]).
$$
\end{prop}
\begin{proof}
According to Proposition~\ref{prop:embedding SP(Lambda) in PDG(Lambda)} it suffices to consider a strictly perfect complex $\cmplx{P}$ of $\Lambda[[G]]$-modules which is also perfect as complex of $\Lambda[[H]]$-modules. Hence, there exists a quasi-isomorphism $\cmplx{Q}\mto \cmplx{P}$ of complexes of $\Lambda[[H]]$-modules with $\cmplx{Q}$ strictly perfect. According to Lemma~\ref{lem:compactflatness property}, each $P^n$ is compact-flat as $\Lambda[[H]]$-module. Therefore, there exists a quasi-isomorphism of complexes of $\Lambda'[[H']]$-modules
$$
\cmplx{(L\tensor_{\Lambda[[H]]}Q)}\wto\cmplx{(L\hat{\tensor}_{\Lambda[[H]]}P)}\wto \cmplx{(K\tensor_{\Lambda[[G]]}P)}.
$$
Since $\cmplx{(L\tensor_{\Lambda[[H]]}Q)}$ is strictly perfect as complex of $\Lambda'[[H']]$-modules, we see that $\cmplx{\ringtransf_{\cmplx{K}}P}$ is in $\cat{PDG}^{\cont,w_{H'}}(\Lambda'[[G']])$.
\end{proof}

\begin{prop}\label{prop:example functors}
The following complexes $\cmplx{K}$ in $\Lambda[[G]]^\op$-$\cat{SP}(\Lambda'[[G']])$ satisfy the hypotheses of Proposition~\ref{prop:localising bimodules}:
\begin{enumerate}
    \item Assume $G=G'$, $H=H'$. For any complex $\cmplx{P}$ in $\Lambda[[G]]^\op$-$\cat{SP}(\Lambda')$ let $\cmplx{K}$ be the complex $\Lambda'[[G]]\tensor_{\Lambda'}\cmplx{P}$ in $\Lambda[[G]]^\op$-$\cat{SP}(\Lambda'[[G]])$ with  the right $G$-operation given by the diagonal action on both factors. This applies in particular for any complex $\cmplx{P}$ in $\Lambda^{\op}$-$\cat{SP}(\Lambda')$ equipped with the trivial $G$-operation.
    \item Assume that $G'$ is an open subgroup of $G$ and set $H'=H\cap G'$. Let $\Lambda=\Lambda'$ and let $\cmplx{K}$ be the complex concentrated in degree 0 given by the $\Lambda[[G']]$-$\Lambda[[G]]$-bimodule $\Lambda[[G]]$.
    \item Assume $\Lambda=\Lambda'$. Let $\alpha\colon G\mto G'$ be a continuous homomorphism such that $\alpha$ maps $H$ to $H'$ and induces a bijection of the sets $H\setminus G$ and $H'\setminus G'$. Let $\cmplx{K}$ be the $\Lambda[[G']]$-$\Lambda[[G]]$-bimodule $\Lambda[[G']]$.
\end{enumerate}
\end{prop}
\begin{proof}
In the first example, one may choose $\cmplx{L}=\Lambda'[[H]]\tensor_{\Lambda'}\cmplx{P}$ with the diagonal right operation of $H$. The isomorphism $\cmplx{L}\hat{\tensor}_{\Lambda[[H]]}\Lambda[[G]]\mto \cmplx{K}$ is then induced by $h\tensor p\tensor g\mapsto hg\tensor pg$ for $h\in H$, $g\in G$ and $p\in P^n$. In the second example, $\cmplx{L}=\Lambda[[H]]$ will do the job. In the last example, choose $\cmplx{L}=\Lambda[[H']]$. The inclusion $\Lambda[[H']]\subset\Lambda[[G']]$ and the continuous ring homomorphism $\Lambda[[G]]\mto \Lambda[[G']]$ induced by $\alpha$ give rise to a morphism of $\Lambda[[H']]$-$\Lambda[[G]]$-bimodules
$$
f\colon\Lambda[[H']]\hat{\tensor}_{\Lambda[[H]]}\Lambda[[G]]\mto\Lambda[[G']].
$$
Let  $U'$ be any open normal subgroup of $G'$, $U=\alpha^{-1}(U)$ its preimage under $\alpha$. Then, as $\Lambda[H'/H'\cap U']$-modules, $\Lambda[H'/H'\cap U']\tensor_{\Lambda[H/H\cap U]}\Lambda[G/U]$ is freely generated by a choice of coset representatives of $UH\setminus G$ and $\Lambda[G'/U']$ is freely generated by the images of these representatives under $\alpha$. Hence, we conclude that $f$ is an isomorphism.
\end{proof}

Generalising \cite[Lemma~2.1]{CFKSV}, we can give a useful characterisation of the complexes in
$\cat{PDG}^{\cont,w_H}(\Lambda[[G]])$ if we further assume that
$H$ is normal in $G$. Under this condition, we find an open
pro-$\ell$-subgroup $K$ in $H$ which is normal in $G$ (take for
example the intersection of all $\ell$-Sylow subgroups of $G$ with
$H$).

\begin{prop}\label{prop:characterisation of relative category}
Let $\Lambda$ be an adic ring. Assume that $H$ is a closed subgroup of $G$ and that both $G$ and $H$ are virtual pro-$\ell$-groups. Let furthermore $K$ be an open
pro-$\ell$-subgroup of $H$ which is normal in $G$. For a complex
$\cmplx{P}=(\cmplx{P}_J)_{J\in\openideals_{\Lambda[[G]]}}$ in
$\cat{PDG}^{\cont}(\Lambda[[G]])$, the following assertions are
equivalent:
\begin{enumerate}
    \item $\cmplx{P}$ is in $\cat{PDG}^{\cont,w_H}(\Lambda[[G]])$
    \item $\ringtransf_{\Lambda/\Jac(\Lambda)[[G/K]]}(\cmplx{P})$ is
    in $\cat{PDG}^{\cont,w_{H/K}}(\Lambda/\Jac(\Lambda)[[G/K]])$,
    \item $\ringtransf_{\Lambda/\Jac(\Lambda)[[G/K]]}(\cmplx{P})$
    has finite cohomology groups.
\end{enumerate}
\end{prop}
\begin{proof}
Assume that $\cmplx{P}$ is a strictly perfect complex of
$\Lambda[[G]]$-modules.
It is a strictly bounded complex of compact-flat
$\Lambda[[H]]$-modules by Lemma~\ref{lem:compactflatness property}. Proposition~\ref{prop:l-admissible group
rings are adic} implies that
$$
\Lambda[[H]]/\Jac(\Lambda[[H]])\tensor_{\Lambda[[H]]}\cmplx{P}
=R/\Jac(R)\tensor_R\Lambda/\Jac(\Lambda)[[G/K]]\tensor_{\Lambda[[G]]}\cmplx{P}
$$
for the finite ring $R=\Lambda/\Jac(\Lambda)[H/K]$. Now the
equivalences in the statement of
Proposition~\ref{prop:characterisation of relative category} are
an immediate consequence of Lemma~\ref{lem:
compactperfect=perfect}.
\end{proof}

We can use similar arguments to prove the following result, which can be combined with Proposition~\ref{prop:localising bimodules}.

\begin{prop}\label{prop:Changing H by an open subgroup}
Let $H$ be a closed subgroup of $G$ such that both $G$ and $H$ are virtual pro-$\ell$-groups. Assume that $H'$ is an open subgroup of $H$. Then
$$
\cat{PDG}^{\cont,w_{H'}}(\Lambda)=\cat{PDG}^{\cont,w_H}(\Lambda).
$$
\end{prop}
\begin{proof}
Since $\Lambda[[H]]$ is a finitely generated free $\Lambda[[H']]$ module, it is clear that every perfect complex of $\Lambda[[H]]$-modules is also perfect as complex of $\Lambda[[H']]$-modules. For the other implication we may shrink $H'$ and assume that it is pro-$\ell$, normal and open in $H$. By Proposition~\ref{prop:l-admissible group rings are adic} we conclude
\begin{align*}
\Lambda[[H']]/\Jac(\Lambda[[H']])&=\Lambda/\Jac(\Lambda),\\ \Lambda[[H]]/\Jac(\Lambda[[H]])&=\left(\Lambda/\Jac(\Lambda)\right)[H/H']/\Jac\left((\Lambda/\Jac(\Lambda))[H/H']\right).
\end{align*}
Let $\cmplx{P}$ be a strictly perfect complex of $\Lambda[[G]]$-modules which is also perfect as complex of $\Lambda[[H']]$-modules. Then Lemma~\ref{lem: compactperfect=perfect} implies that the complexes
$$
\Lambda/\Jac(\Lambda)\tensor_{\Lambda[[H']]}\cmplx{P}\isomorph \Lambda/\Jac(\Lambda)[H/H']\tensor_{\Lambda[[H]]}\cmplx{P}
$$
have finite cohomology groups and that $\cmplx{P}$ is perfect as complex of $\Lambda[[H]]$-modules.
\end{proof}

\section{Perfect Complexes of Adic Sheaves}\label{sec:perf
complexes}

We let $\FF$ denote a finite field of characteristic $p$, with
$q=p^{\nu}$ elements. Furthermore, we fix an algebraic closure
$\algc{\FF}$ of $\FF$.

For any scheme $X$ in the category $\cat{Sch}_{\FF}$ of
 $\FF$-schemes of finite type and any adic ring $\Lambda$
we introduced in \cite{Witte:PhD} a Waldhausen category
$\cat{PDG}^{\cont}(X,\Lambda)$ of perfect complexes of adic
sheaves on $X$. Below, we will recall the definition.

\begin{defn}
Let $R$ be a finite ring and $X$ be a scheme in
$\cat{Sch}_{\FF}$. A complex $\cmplx{\sheaf{F}}$ of \'etale
sheaves of left $R$-modules on $X$ is called \emph{strictly
perfect} if it is strictly bounded and each $\sheaf{F}^n$ is
constructible and flat. A complex is called \emph{perfect} if it
is quasi-isomorphic to a strictly perfect complex. It is
\emph{$DG$-flat} if for each geometric point of $X$, the complex
of stalks is $DG$-flat.
\end{defn}

\begin{defn}\label{defn:PDGcont(X,Lambda)}
Let $X$ be a scheme in $\cat{Sch}_{\FF}$ and let $\Lambda$ be an
adic ring. The \emph{category of perfect complexes of adic
sheaves} $\cat{PDG}^{\cont}(X,\Lambda)$ is the following
Waldhausen category. The objects of $\cat{PDG}^{\cont}(X,\Lambda)$
are inverse systems $(\cmplx{\sheaf{F}}_I)_{I\in
\openideals_{\Lambda}}$ such that:
\begin{enumerate}
\item for each $I\in\openideals_{\Lambda}$, $\cmplx{\sheaf{F}}_I$
is a perfect and $DG$-flat complex of \'etale sheaves of $\Lambda/I$-modules on $X$,

\item for each $I\subset J\in\openideals_{\Lambda}$, the
transition morphism
$$
\varphi_{IJ}:\cmplx{\sheaf{F}}_I\mto \cmplx{\sheaf{F}}_J
$$
of the system induces an isomorphism
$$
\Lambda/J\tensor_{\Lambda/I}\cmplx{\sheaf{F}}_I\wto
\cmplx{\sheaf{F}}_J.
$$
\end{enumerate}
Weak equivalences and cofibrations are those morphisms of inverse
systems that are weak equivalences or cofibrations for each
$I\in\openideals_{\Lambda}$, respectively.
\end{defn}

If $\Lambda=\Int_{\ell}$, then the subcategory of complexes
concentrated in degree $0$ of $\cat{PDG}^{\cont}(X,\Int_{\ell})$
corresponds precisely to the exact category of flat constructible
$\ell$-adic sheaves on $X$ in the sense of \cite[Expos\'e~VI,
Definition~1.1.1]{SGA5}. In this sense, we recover the classical
theory.

If $f\colon Y\mto X$ is a morphism of schemes, we define a
Waldhausen exact functor
$$
f^*\colon\cat{PDG}^{\cont}(X,\Lambda)\mto\cat{PDG}^{\cont}(Y,\Lambda),\qquad
(\cmplx{\sheaf{F}}_I)_{I\in\openideals_{\Lambda}}\mapsto
(f^*\cmplx{\sheaf{F}}_I)_{I\in\openideals_{\Lambda}}.
$$
We will also need a Waldhausen exact functor that computes higher
direct images with proper support. For the purposes of this
article it suffices to use the following construction.

\begin{defn}
Let $f\colon X\mto Y$ be a morphism of separated schemes in $\cat{Sch}_{\FF}$.
Then there exists a factorisation $f=p\comp j$ with $j\colon
X\oimm X'$ an open immersion and $p\colon X'\mto Y$ a proper
morphism. Let $\cmplx{\God}_{X'}\sheaf{K}$ denote the Godement
resolution of a complex $\cmplx{\sheaf{K}}$ of abelian \'etale
sheaves on $X'$. Define
\begin{align*}
\RDer f_!\colon \cat{PDG}^{\cont}(X,\Lambda)&\mto
\cat{PDG}^{\cont}(Y,\Lambda)\\
(\cmplx{\sheaf{F}}_I)_{I\in\openideals_{\Lambda}}&\mapsto(f_*\cmplx{\God}_{X'}
 j_!\sheaf{F}_I)_{I\in\openideals_{\Lambda}}
\end{align*}
\end{defn}

Note that this definition depends on the particular choice of the
compactification $f=p\comp j$. However, all possible choices will
induce the same homomorphisms
$$
\KTh_n(\RDer f_!)\colon \KTh_n(\cat{PDG}^{\cont}(X,\Lambda))\mto
\KTh_n(\cat{PDG}^{\cont}(Y,\Lambda))
$$
and this is all we need.

\begin{defn}
Let $X$ be a separated scheme in $\cat{Sch}_{\FF}$ and write $h\colon
X\mto \Spec \FF$ for the structure map, $s\colon \Spec
\algc{\FF}\mto \Spec \FF$ for the map induced by the embedding
into the algebraic closure. We define the Waldhausen exact functors
$$
\RDer\Sectc(\algc{X},-),\RDer\Sectc(X,-)\colon\cat{PDG}^{\cont}(X,\Lambda)\mto\cat{PDG}^{\cont}(\Lambda)
$$
to be the composition of
$$
\RDer h_!\colon \cat{PDG}^{\cont}(X,\Lambda)\mto
\cat{PDG}^{\cont}(\Spec \FF,\Lambda)
$$
with the section functors
\begin{align*}
\cat{PDG}^{\cont}(\Spec
\FF,\Lambda)&\mto\cat{PDG}^{\cont}(\Lambda), \\
(\cmplx{\sheaf{F}}_{I})_{I\in\openideals_{\Lambda}}&\mto
(\Sect(\Spec
\algc{\FF},s^*\cmplx{\sheaf{F}}_I))_{I\in\openideals_{\Lambda}},\\
(\cmplx{\sheaf{F}}_{I})_{I\in\openideals_{\Lambda}}&\mto
(\Sect(\Spec
\FF,\cmplx{\sheaf{F}}_I))_{I\in\openideals_{\Lambda}},
\end{align*}
respectively.
\end{defn}

\begin{defn}
We let $\Frob_{\FF}\in\Gal(\algc{\FF}/\FF)$ denote the \emph{geometric Frobenius} of $\FF$, i.\,e.\ if $\FF$ has $q$ elements and $x\in\algc{\FF}$, then $\Frob_{\FF}(x)=x^{\frac{1}{q}}$.
\end{defn}

Clearly, $\Frob_{\FF}$ operates on $\RDer\Sectc(\algc{X},\cmplx{\sheaf{F}})$.

\begin{prop}\label{prop:fundamental exact sequence}
Let $X$ be a separated scheme in $\cat{Sch}_{\FF}$. The following sequence is exact in $\cat{PDG}^{\cont}(X,\Lambda)$:
$$
0\mto\RDer\Sectc(X,\cmplx{\sheaf{F}})\mto\RDer\Sectc(\algc{X},\cmplx{\sheaf{F}})\xrightarrow{\id-\Frob_{\FF}}\RDer\Sectc(\algc{X},\cmplx{\sheaf{F}})\mto 0.
$$
\end{prop}
\begin{proof}
See \cite[Proposition~6.1.2]{Witte:PhD}.
(In fact, all that we will need later on is that the cone of $\id-\Frob_{\FF}$ is quasi-isomorphic to $\RDer\Sectc(X,\cmplx{\sheaf{F}})$ shifted by one, which is a well-known consequence of the Hochschild-Serre spectral sequence.)
\end{proof}

The definition of $\ringtransf_{\cmplx{M}}$ extends to
$\cat{PDG}^{\cont}(X,\Lambda)$ as follows.

\begin{defn}
For two adic rings $\Lambda$ and $\Lambda'$ we
let $\Lambda^{\op}$-$\cat{SP}(X,\Lambda')$ the Waldhausen category of strictly bounded complexes $(\cmplx{\sheaf{F}}_J)_{J\in\openideals_{\Lambda'}}$ in $\cat{PDG}^{\cont}(X,\Lambda')$ with each $\sheaf{F}_{J}^n$ a sheaf of $\Lambda'/J$-$\Lambda$-bimodules, constructible and flat as sheaf of $\Lambda'/J$-modules. The transition maps in the system $(\cmplx{\sheaf{F}}_J)_{J\in\openideals_{\Lambda'}}$ and the boundary maps of the complexes are supposed to be compatible with the right $\Lambda$-structure.
\end{defn}

In particular, a complex $\cmplx{M}$ in $\Lambda^\op$-$\cat{SP}(\Lambda')$ can be identified with the complex of constant sheaves $(\Lambda'/I\tensor_{\Lambda'}M)_{I\in \openideals_{\Lambda'}}$ in $\Lambda^\op$-$\cat{SP}(X,\Lambda')$.

\begin{defn}
For $(\cmplx{\sheaf{F}}_I)_{I\in\openideals_{\Lambda}}\in
\cat{PDG}^{\cont}(X,\Lambda)$ and $\cmplx{\sheaf{K}}\in
\Lambda^{\op}\text{-}\cat{SP}(X,\Lambda')$ we set
$$
\ringtransf_{\cmplx{\sheaf{K}}}\left((\cmplx{\sheaf{F}}_I)_{I\in\openideals_{\Lambda}}\right)=
(\varprojlim_{J\in\openideals_{\Lambda}}
\cmplx{(\sheaf{K}_I\tensor_{\Lambda}\sheaf{F}_{J})})_{I\in\openideals_{\Lambda'}}
$$
and obtain a Waldhausen exact functor
$$
\ringtransf_{\cmplx{\sheaf{K}}}\colon
\cat{PDG}^{\cont}(X,\Lambda)\mto\cat{PDG}^{\cont}(X,\Lambda').
$$
\end{defn}

\begin{prop}\label{prop:compatibility of change of rings}
Let $X$ be a separated scheme in $\cat{Sch}_{\FF}$ and let $\cmplx{M}$ be a complex in
$\Lambda^{\op}\text{-}\cat{SP}(\Lambda')$. The natural morphisms
\begin{align*}
\ringtransf_{\cmplx{M}}\RDer\Sectc(X,\cmplx{\sheaf{F}})&\mto \RDer\Sectc(X,\ringtransf_{\cmplx{M}}\cmplx{\sheaf{F}})\\
\ringtransf_{\cmplx{M}}\RDer\Sectc(\algc{X},\cmplx{\sheaf{F}})&\mto \RDer\Sectc(\algc{X},\ringtransf_{\cmplx{M}}\cmplx{\sheaf{F}})
\end{align*}
are quasi-isomorphisms.
\end{prop}
\begin{proof}
This is straightforward. See \cite[Proposition~5.5.7]{Witte:PhD}.
\end{proof}

\section{Adic Sheaves Induced by Coverings}\label{sec:adic sheaves induced by coverings}

As before, we let $\FF$ be a finite field
and $X$ an $\FF$-scheme of finite type. Recall that for any finite \'etale map $h\colon Y\mto X$ and any abelian \'etale sheaf $\sheaf{F}$ on $X$, $h_!h^*\sheaf{F}$ is the sheaf associated to the presheaf
$$
U\mapsto\bigoplus_{\varphi\in\Hom_X(U,Y)}\sheaf{F}(U)
$$
with the transition maps
\begin{align*}
\bigoplus_{\psi\in \Hom_X(V,Y)}\sheaf{F}(V)&\mto \bigoplus_{\varphi\in \Hom_X(U,Y)}\sheaf{F}(U),\\
(x_\psi)&\mapsto \big(\sum_{\varphi=\psi\comp \alpha}\sheaf{F}(\alpha)(x_\psi)\big)
\end{align*}
for $\alpha \colon U\mto V$. If $h$ is a finite principal covering with Galois group $G$, then the right action of $G$ on $Y$ induces a right action on $\Hom_X(U,Y)$ and hence, a left action on $h_!h^*\sheaf{F}$ by permutation of the components. The stalk at a geometric point $\xi$ of $X$ is given by
$$
(h_!h^*\sheaf{F})_{\xi}=\bigoplus_{\varphi\in\Hom_X(\xi,Y)}\sheaf{F}_{\xi}.
$$
Since $Y$ is finite over $X$, the set $\Hom_X(\xi,Y)$ is nonempty. The choice of any element in $\Hom_X(\xi,Y)$ induces an isomorphism of $G$-sets
$$
\Hom_X(\xi,Y)\isomorph \Hom_Y(\xi,Y\times_X Y)\isomorph \Hom_Y(\xi,\bigsqcup_{g\in G} Y)\isomorph G
$$
and hence, a $\Int[G]$-isomorphism
$$
(h_!h^*\sheaf{F})_\xi\isomorph \Int[G]\tensor_{\Int}\sheaf{F}_\xi.
$$

Consider an adic
$\Int_{\ell}$-algebra $\Lambda$ and let $(f\colon Y\mto
X,G)$ be a virtual pro-$\ell$-covering of $X$.

\begin{defn}
For $\cmplx{\sheaf{F}}\in \cat{PDG}^{\cont}(X,\Lambda)$ we set
$$
f_!f^*\cmplx{\sheaf{F}}=(\varprojlim_{I\in\openideals_{\Lambda}}\varprojlim_{U\in
\NS_{G}}\Lambda[[G]]/J\tensor_{\Lambda[[G]]}{f_U}_!f_U^*\cmplx{\sheaf{F}_I})_{J\in\openideals_{\Lambda[[G]]}}
$$
\end{defn}

Again, we note that for each $J\in \openideals_{\Lambda[[G]]}$, there exists
an $I_0\in \openideals_{\Lambda}$ and an $U_0\in\NS_{G}$ such
that $\Lambda[[G]]/J$ is a right $\Lambda/I_0[G/U_0]$-module and
such that
$$
\cmplx{(f_!f^*\sheaf{F})}_J\isomorph
\Lambda[[G]]/J\tensor_{\Lambda/I_0[G/U_0]}f_{U_0!}f_{U_0}^*\cmplx{\sheaf{F}}_{I_0}.
$$

\begin{prop}
For any complex $\cmplx{\sheaf{F}}$ in
$\cat{PDG}^{\cont}(X,\Lambda)$, $f_!f^*\cmplx{\sheaf{F}}$ is a
complex in $\cat{PDG}^{\cont}(X,\Lambda[[G]])$. Moreover, the
functor
$$
f_!f^*\colon \cat{PDG}^{\cont}(X,\Lambda)\mto
\cat{PDG}^{\cont}(X,\Lambda[[G]])
$$
is Waldhausen exact.
\end{prop}
\begin{proof}
We note that ${f_{U}}_!f_{U}^*\cmplx{\sheaf{F}}_I$ is a perfect $DG$-flat complex of sheaves of $\Lambda/I[G/U]$-modules. This follows since for every
geometric point $\xi$ of $X$ and every \'etale sheaf of left
$\Lambda/I$-modules $\sheaf{P}$ on $X$, we have
$$
(f_{U!}f_{U}^*\sheaf{P})_{\xi}\isomorph
\Lambda/I[G/U]\tensor_{\Lambda}(\sheaf{P}_{\xi})
$$
Moreover, the functor $f_{U!}f_{U}^*$ is exact as functor from the
abelian category of sheaves of $\Lambda/I$-modules to the abelian
category of sheaves of $\Lambda/I[G/U]$-modules and for $V\subset
U$, $J\subset I$, we have a natural isomorphism of functors
$$
\Lambda/I[G/U]\tensor_{\Lambda/J[G/V]}f_{V!}f_V^*\isomorph
f_{U!}f_U^*.
$$
These observations suffice to deduce the assertion.
\end{proof}

Sometimes, the following alternative description of the functor $f_!f^*$ is useful.

\begin{prop}\label{prop:alternative description of f_!f^*}
The sheaf $(f_!f^*\Lambda)$ is in $\Lambda^{\op}$-$\cat{SP}(X,\Lambda[[G]])$ and for any $\cmplx{\sheaf{F}}$ in $\cat{PDG}^{\cont}(X,\Lambda)$, there exists a natural isomorphism
$$
\ringtransf_{f_!f^*\Lambda}(\cmplx{\sheaf{F}})\isomorph f_!f^*\cmplx{\sheaf{F}}
$$
in $\cat{PDG}^{\cont}(X,\Lambda[[G]])$
\end{prop}
\begin{proof}
One easily reduces to the case that $\Lambda$ and $G$ are finite. Then the isomorphism is provided by the well-known projection formula:
$$
f_!f^*\Lambda\tensor_{\Lambda}\cmplx{\sheaf{F}}\isomorph f_!(f^*\Lambda\tensor_{\Lambda}f^*\cmplx{\sheaf{F}})\isomorph f_!f^*\cmplx{\sheaf{F}}.
$$
\end{proof}

In the following three propositions we formulate various base change compatibilities of our construction.

\begin{prop}[Change of the base scheme.]\label{prop:base change for coverings}
Let $a\colon X'\mto X$ be a morphism of separated schemes in $\cat{Sch}_{\FF}$ and write $(f'\colon Y'\mto X',G)$ for the principal covering obtained from $(f\colon Y\mto X,G)$ by base change. Then
\begin{enumerate}
\item For any $\cmplx{\sheaf{F}}$ in $\cat{PDG}^{\cont}(X,\Lambda)$ there is a natural isomorphism
$$f'_!f'^*a^*\cmplx{\sheaf{F}}\isomorph a^*f_!f^*\cmplx{\sheaf{F}}$$
in $\cat{PDG}^{\cont}(X',\Lambda[[G]])$.
\item For any $\cmplx{\sheaf{F}}$ in $\cat{PDG}^{\cont}(X',\Lambda[[G]])$ there is a natural quasi-isomorphism
$$
f_!f^*\RDer a_!\cmplx{\sheaf{F}}\wto \RDer a_!f'_!f'^*\cmplx{\sheaf{F}}
$$
in $\cat{PDG}^{\cont}(X,\Lambda[[G]])$.
\end{enumerate}
\end{prop}
\begin{proof}
The first assertion follows from an application of the proper base change theorem in a very trivial case. For the second assertion, we use the projection formula to see that the natural morphism
$$
\ringtransf_{f_!f^*\Lambda}\RDer a_!\cmplx{\sheaf{F}}\mto \RDer a_! \ringtransf_{a^*f_!f^*\Lambda}\cmplx{\sheaf{F}}
$$
is a quasi-isomorphism and then the first assertion to identify $a^*f_!f^*\Lambda$ with $f'_!f'^*\Lambda$.
\end{proof}

\begin{prop}[Change of the group.]\label{prop:ring change for coverings}
Let $\cmplx{\sheaf{F}}$ be a complex in $\cat{PDG}^{\cont}(X,\Lambda)$.
\begin{enumerate}
\item Let $H$ be a closed normal subgroup of $G$. Then there
exists a natural isomorphism
$$
\ringtransf_{\Lambda[[G/H]]}f_!f^*\cmplx{\sheaf{F}}\wto (f_H)_!(f_H)^*\cmplx{\sheaf{F}}
$$
in
$\cat{PDG}^{\cont}(X,\Lambda[[G/H]])$.

\item Let $U$ be an open subgroup of $G$, let $f_U\colon Y_U\mto
X$ denote the natural projection map, and view $\Lambda[[G]]$ as a
$\Lambda[[U]]$-$\Lambda[[G]]$-bimodule. Then there exists a
natural quasi-isomorphism
$$
\ringtransf_{\Lambda[[G]]}f_!f^*\cmplx{\sheaf{F}}\wto \left(\RDer
(f_U)_!\right)\left((f^U)_!(f^U)^*\right)f_U^*\cmplx{\sheaf{F}}
$$
in
$\cat{PDG}^{\cont}(X,\Lambda[[U]])$.
\end{enumerate}
\end{prop}
\begin{proof}
One reduces to the case that $\Lambda$ and $G$ are finite and that $\cmplx{\sheaf{F}}=\Lambda$. The first morphism is induced by the natural map $f_!f^*\Lambda\mto (f_H)_!(f_H)^*\Lambda$ and is easily checked to be an isomorphism by looking at the stalks. The second morphism is the composition of the isomorphism $f_!f^*\Lambda\isomorph {f_U}_!f^U_!{f^U}^*f_U^*\Lambda$ with the functorial morphism ${f_U}_!\mto \RDer {f_U}_!$, the latter being a quasi-isomorphism since $f_U$ is finite.
\end{proof}

\begin{defn}\label{defn:induced complexes}
Let $\Lambda$ and $\Lambda'$ be two adic $\Int_{\ell}$-algebras and let $\cmplx{\sheaf{K}}$ be in $\Lambda[[G]]^{\op}$-$\cat{SP}(X,\Lambda')$.
\begin{enumerate}
\item
We will write $\cmplx{\sheaf{K}[[G]]^{\delta}}$ for the complex $\ringtransf_{\Lambda'[[G]]}\cmplx{\sheaf{K}}$ in $\Lambda[[G]]^{\op}$-$\cat{SP}(X,\Lambda'[[G]])$  with the right $\Lambda[[G]]$-structure given by the diagonal right operation of $G$.

\item We will write $\cmplx{\widetilde{\sheaf{K}}}$ for the complex $\ringtransf_{\cmplx{\sheaf{K}}}f_!f^*\Lambda$ in $\Lambda^{\op}$-$\cat{SP}(X,\Lambda')$.
\end{enumerate}
\end{defn}

\begin{prop}[Compatibility with tensor products.]\label{prop:ring change for coverings II}
Let $\cmplx{\sheaf{K}}$ be in $\Lambda[[G]]^{\op}$-$\cat{SP}(X,\Lambda')$. For every $\cmplx{\sheaf{F}}$ in $\cat{PDG}^{\cont}(X,\Lambda)$ there exists a natural isomorphism
$$
\ringtransf_{\cmplx{\sheaf{K}[[G]]^{\delta}}}f_!f^*\cmplx{\sheaf{F}}\isomorph f_!f^*\ringtransf_{\cmplx{\widetilde{\sheaf{K}}}}\cmplx{\sheaf{F}}
$$
\end{prop}
\begin{proof}
One easily reduces to the case that $\Lambda'$ and $\Lambda$ are finite rings and that $G$ is a finite group. Moreover, it suffices to consider a sheaf $\sheaf{K}$ of $\Lambda[G]$-$\Lambda$-bimodules viewed as a complex in $\Lambda[G]^\op$-$\cat{SP}(X,\Lambda')$ which is concentrated in degree $0$. In view of Proposition~\ref{prop:alternative description of f_!f^*} we may also assume that $\cmplx{\sheaf{F}}=\Lambda$. We begin by proving two special cases.

\emph{Case 1.} Assume that $G$ operates trivially on $\sheaf{K}$. Then
$$
\ringtransf_{\sheaf{K}[[G]]^{\delta}}f_!f^*\Lambda\isomorph\sheaf{K}\tensor_{\Lambda}f_!f^*\Lambda\isomorph f_!f^*\sheaf{K}
$$
by the projection formula. On the other hand,
$$
\ringtransf_{\widetilde{\sheaf{K}}}\Lambda\isomorph\sheaf{K}\tensor_{\Lambda[G]}f_!f^*\Lambda\isomorph\sheaf{K},
$$
and therefore, $\ringtransf_{\sheaf{K}[[G]]^{\delta}}f_!f^*\Lambda\isomorph f_!f^*\ringtransf_{\widetilde{\sheaf{K}}}\Lambda$.

\emph{Case 2.} Assume that $\Lambda'=\Lambda[G]$ and that $\sheaf{K}$ is the constant sheaf $\Lambda[G]$. Let $ U\mto X$ be finite \'etale and consider the homomorphism
\begin{align*}
u\colon \bigoplus_{\psi\in \Hom_X(U,Y)}\Lambda&\mto \bigoplus_{\psi\in\Hom_X(U,Y)}\bigoplus_{\phi\in\Hom_X(U,Y)}\Lambda\\
(a_\psi)&\mapsto (a_\psi\delta_{\psi,\phi})
\end{align*}
with
$$
\delta_{\psi,\phi}=
\begin{cases}
  1&\text{if $\psi =\phi$,}\\
  0&\text{else.}
\end{cases}
$$
For $g\in G$ we have
$$
u(g(a_\psi))=(g,g)u(a_\psi).
$$
Hence, $u$ induces a $\Lambda[G][G]$-homomorphism
$$
\ringtransf_{\Lambda[G][G]^{\delta}}(f_!f^*\Lambda)=\Lambda[G][G]^\delta\tensor_{\Lambda}f_!f^*\Lambda\mto f_!f^*f_!f^*\Lambda\isomorph f_!f^*\ringtransf_{\widetilde{\Lambda[G]}}(\Lambda)
$$
which is seen to be an isomorphism by checking on the stalks.

To prove the general case, we let $\sheaf{K}'$ be the sheaf $\sheaf{K}$ considered as a sheaf of $\Lambda'$-$\Lambda[G][G]$-bimodules, where the operation of the second copy of $G$ is the trivial one. Then we have an obvious isomorphism of sheaves of $\Lambda'[G]$-$\Lambda[G]$-bimodules
$$
\sheaf{K}[G]^\delta\isomorph \sheaf{K'}[G]^\delta\tensor_{\Lambda[G][G]}\Lambda[G][G]^\delta
$$
and by the two cases that we have already proved we obtain
\begin{align*}
f_!f^*\ringtransf_{\widetilde{\sheaf{K}}}(\Lambda)&\isomorph f_!f^*\ringtransf_{\widetilde{\sheaf{K}}'}f_!f^*\Lambda
\isomorph \ringtransf_{\sheaf{K}'[G]^\delta}f_!f^*f_!f^*\Lambda\\
&\isomorph \ringtransf_{\sheaf{K}'[G]^\delta}\ringtransf_{\Lambda[G][G]^\delta}f_!f^*\Lambda\isomorph\ringtransf_{\sheaf{K}[G]^\delta}f_!f^*\Lambda
\end{align*}
as desired.
\end{proof}

The significance of the construction in \ref{defn:induced complexes}.(2) is partly explained by
the following version of the well-known equivalence of finite representations of the fundamental group with locally constant \'etale sheaves on a connected scheme $X$.

\begin{prop}\label{prop:locally constant sheaves}
Let $\Lambda$ and $\Lambda'$ be two adic $\Int_{\ell}$-algebras. Assume that $X$ is connected and that $x$ is a geometric point of $X$. Let $(f\colon Y\mto X,G)$ be a virtual pro-$\ell$-subcovering of the universal covering $(\widetilde{X}\mto X,\fundG(X,x))$. The functor
$$
\Lambda[[G]]^\op\text{-}\cat{SP}(\Lambda')\mto \Lambda^{\op}\text{-}\cat{SP}(X,\Lambda'),\qquad \cmplx{P}\mapsto\cmplx{\widetilde{P}}
$$
identifies $\Lambda[[G]]^\op\text{-}\cat{SP}(\Lambda')$ with a full subcategory $\cat{C}$ of $\Lambda^{\op}\text{-}\cat{SP}(X,\Lambda')$. The objects of $\cat{C}$ are systems of strictly bounded complexes $(\cmplx{\sheaf{F}}_{I})_{I\in\openideals_{\Lambda}}$ of sheaves of $\Lambda'/I$-$\Lambda$-bimodules such that for each $n$, $\sheaf{F}_I^n$ is constructible and flat as sheaf of $\Lambda'/I$-modules and there exists an open normal subgroup $U$ of $G$ such that $f_U^*\sheaf{F}^n_I$ is a constant sheaf.
\end{prop}
\begin{proof}
We may assume that $\Lambda'$ is finite. Clearly, $\cmplx{\widetilde{P}}$ is an object of $\cat{C}$ for every complex $\cmplx{P}$ in $\Lambda^{\op}$-$\cat{SP}(\Lambda')$. To construct the inverse functor, we fix a compatible family of geometric points $(x\mto Y_U)_{U\in\NS_G}$. If $\cmplx{\sheaf{F}}$ is in $\cat{C}$, there exists an open normal subgroup $U$ of $G$ such that the induced morphism $\cmplx{\sheaf{F}(Y_U)}\mto\cmplx{\sheaf{F}}_x$ is an isomorphism. Turn $\cmplx{\sheaf{F}}_x$ into a complex in $\Lambda[[G]]^\op$-$\cat{SP}(\Lambda')$ by considering the contragredient of the left action of $G$ on $\cmplx{\sheaf{F}(Y_U)}$. One then checks that this is an inverse to the functor $\cmplx{P}\mapsto\cmplx{\widetilde{P}}$.
\end{proof}

\begin{rem}
Extending Definition~\ref{defn:PDGcont(X,Lambda)} to arbitrary compact rings, one can also prove a corresponding statement for the full universal covering. On the other hand, if $\Lambda$ and $\Lambda'$ are adic rings, it follows as in \cite[Theorem~5.6.5]{Witte:PhD} that for every $\cmplx{K}$ in $\Lambda[[\fundG(X,x)]]^\op$-$\cat{SP}(\Lambda')$ there exists a factor group $G$ of $\fundG(X,x)$ such that $G$ is a virtual pro-$\ell$-group and such that $\cmplx{K}$ also lies in $\Lambda[[G]]^\op$-$\cat{SP}(\Lambda')$.
\end{rem}

\begin{rem}
Proposition~\ref{prop:locally constant sheaves} implies in particular that for any virtual pro-$\ell$-subcovering $(f\colon Y\mto X,G)$ of the universal covering, the sheaf $\sheaf{M}(G)$ of the introduction corresponds to $f_!f^*\Int_{\ell}$.
\end{rem}

\section{The Cyclotomic \texorpdfstring{$\Gamma$-}{}Covering}\label{sec:the cyclotomic covering}

Let $X$ be a separated scheme in $\cat{Sch}_{\FF}$. For any complex
$\cmplx{\sheaf{F}}=(\cmplx{\sheaf{F}}_I)_{I\in\openideals_{\Lambda}}$ in
$\cat{PDG}^{\cont}(X,\Lambda)$, we write
$$
\HF_c^i(X,\cmplx{\sheaf{F}})=\HF^i(\varprojlim_{I\in\openideals_{\Lambda}}\RDer\Sectc(X,\cmplx{\sheaf{F}_I}))
$$
for the $i$-th hypercohomology module of the complex
$\RDer\Sectc(X,\cmplx{\sheaf{F}})$.

\begin{prop}\label{prop:the Zl case}
Let $(f\colon X_{\ell^{\infty}}\mto X,\Gamma_{\ell^{\infty}})$ be the cyclotomic
$\Gamma_{\ell^{\infty}}$-covering of $X$. For all $i\in \Int$ and any complex
$\cmplx{\sheaf{F}}=(\cmplx{\sheaf{F}}_I)_{I\in\openideals_{\Lambda}}$
in $\cat{PDG}^{\cont}(X,\Lambda)$, we have
$$
\HF_c^i(X,f_!f^*\cmplx{\sheaf{F}})\isomorph\varprojlim_{n}\HF_c^{i-1}(\algc{X},\cmplx{\sheaf{F}})/(\id-\Frob_\FF^{\ell^n})\HF_c^{i-1}(\algc{X},\cmplx{\sheaf{F}})
$$
as $\Lambda$-modules.
\end{prop}
\begin{proof}
Write $f_n=f_{\Gamma_{\ell^{\infty}}^n}$. Since
$$
\RDer^1\varprojlim_{J\in\openideals_{\Lambda[[\Gamma_{\ell^{\infty}}]]}}M_J=0
$$
for any inverse system $(M_J)_{J\in\openideals_{\Lambda[[\Gamma_{\ell^{\infty}}]]}}$ of $\Lambda[[\Gamma_{\ell^{\infty}}]]$-modules with surjective transition maps and since
the cohomology groups
$\HF_c^i(X,{f_{n}}_!f_{n}^*\cmplx{\sheaf{F}}_I)$ are finite for $I\in\openideals_{\Lambda}$, we conclude that
$$
\HF_c^i(X,f_!f^*\cmplx{\sheaf{F}})
\isomorph\varprojlim_{I\in\openideals_{\Lambda}}\varprojlim_{n}\HF_c^i(X,{f_{n}}_!f_{n}^*\cmplx{\sheaf{F}}_I)
$$
(see also \cite[Proposition~5.3.2]{Witte:PhD}).

Moreover, for every $n$, there is a commutative diagram with exact
rows
$$
\xymatrix{
\cdots\ar[r]&\HF_c^i(X,{f_{n+1}}_!f_{n+1}^*\cmplx{\sheaf{F}}_I)\ar[d]^{tr}\ar[r]&
\HF_c^i(\algc{X},\cmplx{\sheaf{F}}_I)\ar[d]^{\sum\limits_{k=0}^{\ell-1}\Frob_\FF^{k\ell^n}}\ar[r]^{\id-\Frob_{\FF}^{\ell^{n+1}}}&
\HF_c^i(\algc{X},\cmplx{\sheaf{F}}_I)\ar[r]\ar[d]^{=}&\dots\\
\cdots\ar[r]&\HF_c^i(X,{f_{n}}_!f_{n}^*\cmplx{\sheaf{F}}_I)\ar[r]&
\HF_c^i(\algc{X},\cmplx{\sheaf{F}}_I)\ar[r]^{\id-\Frob_{\FF}^{\ell^{n}}}&
\HF_c^i(\algc{X},\cmplx{\sheaf{F}}_I)\ar[r]&\dots\\
}
$$
where
$tr\colon{f_{n+1}}_!f_{n+1}^*\cmplx{\sheaf{F}}_I
\mto{f_{n}}_!f_{n}^*\cmplx{\sheaf{F}}_I$
denotes the usual trace map. Set
$$
K_n=\ker\left(\HF_c^i(\algc{X},\cmplx{\sheaf{F}}_I)
\xrightarrow{\id-\Frob_{\FF}^{\ell^{n}}}
\HF_c^i(\algc{X},\cmplx{\sheaf{F}}_I)\right)
$$
Since $\HF_c^i(\algc{X},\cmplx{\sheaf{F}}_I)$ is a finite group, the inclusion
chain
$$
K_0\subset K_{1}\subset \dotso\subset K_{n}\subset\dotso
$$
becomes stationary. Hence, for large $n$, $K_n=K_{n+1}$ and the map
$$
\sum_{k=0}^{\ell-1}\Frob_{\FF}^{k\ell^n}\colon K_{n+1}\mto K_{n}
$$
is multiplication by $\ell$. Since $K_n$ is annihilated by a power of $\ell$, we conclude
$$
\varprojlim_{n} K_n=0.
$$
The equality claimed in the proposition is an immediate
consequence.
\end{proof}

\begin{prop}\label{prop:twisted Frob operation}
Let $\gamma$ denote the image of $\Frob_{\FF}$ in $\Gamma=\Gamma_{k\ell^{\infty}}$ and let $(X_{k\ell^{\infty}}\mto X,\Gamma)$ be the cyclotomic $\Gamma$-covering of $X$. Let $\cmplx{\sheaf{F}}$ be a complex in $\cat{PDG}^{\cont}(X,\Lambda)$. There exists a quasi-isomorphism
$$
\eta\colon \ringtransf_{\Lambda[[\Gamma]]}\RDer\Sectc(\algc{X},\cmplx{\sheaf{F}})\mto \RDer\Sectc(\algc{X},f_!f^*\cmplx{\sheaf{F}}) $$
in $\cat{PDG}^{\cont}(\Lambda[[\Gamma]])$ such that the following diagram commutes:
$$
\xymatrix{
\ringtransf_{\Lambda[[\Gamma]]}\RDer\Sectc(\algc{X},\cmplx{\sheaf{F}})\ar[r]^{\gamma^{-1}\tensor \Frob_{\FF}}\ar[d]^{\eta}&
\ringtransf_{\Lambda[[\Gamma]]}\RDer\Sectc(\algc{X},\cmplx{\sheaf{F}})\ar[d]^{\eta}\\
 \RDer\Sectc(\algc{X},f_!f^*\cmplx{\sheaf{F}})\ar[r]^{\Frob_{\FF}}&
  \RDer\Sectc(\algc{X},f_!f^*\cmplx{\sheaf{F}})
}
$$
\end{prop}
\begin{proof}
Using Proposition~\ref{prop:base change for coverings} we can reduce to the case $X=\Spec \FF$. Moreover, it suffices to consider $\cmplx{\sheaf{F}}=\Lambda$. By Proposition~\ref{prop:locally constant sheaves}, the sheaf $f_!f^*\Lambda$ corresponds to the $\Lambda[[\Gamma]]$-module $\Lambda[[\Gamma]]$ with the left action of the Frobenius $\Frob_{\FF}$ given by right multiplication with $\gamma^{-1}$. The assertion of the proposition is an immediate consequence.
\end{proof}

\section{The Iwasawa Main Conjecture}\label{sec:IMC}

The following theorem is the central piece of our analogue of the noncommutative Iwasawa main conjecture. It corresponds to \cite[Conjecture~5.1]{CFKSV} in conjunction with the vanishing of the $\mu$-invariant, the complex $\RDer\Sectc(X,f_!f^*\cmplx{\sheaf{F}})$ playing the role of the module $X(E/F_{\infty})$.

\begin{thm}\label{thm:S-torsion property}
Let $X$ be a separated scheme of finite type over a finite field
$\FF$. Fix a prime $\ell$ and let $(f\colon Y\mto X,G)$ be an admissible covering of
$X$ with group $G\isomorph H\rtimes \Gamma_{\ell^{\infty}}$. If $\Lambda$ is an adic
$\Int_{\ell}$-algebra and $\cmplx{\sheaf{F}}$ a complex in
$\cat{PDG}^{\cont}(X,\Lambda)$, then
$\RDer\Sectc(X,f_!f^*\cmplx{\sheaf{F}})$ is in
$\cat{PDG}^{\cont,w_H}(\Lambda[[G]])$.
\end{thm}
\begin{proof}
Proposition~\ref{prop:ring change for coverings}, Proposition~\ref{prop:ring change for coverings II} (for the $\Lambda/I$-$\Lambda$-bimodule $\Lambda/I$ with trivial $G$-operation), and
Proposition~\ref{prop:compatibility of change of rings} imply that for any
closed normal subgroup $K$ of $G$ and each open two-sided ideal
$I$ of $\Lambda$, there exists a quasi-isomorphism
\begin{align*}
\ringtransf_{\Lambda/I[[G/K]]}\RDer\Sectc(X,f_!f^*\cmplx{\sheaf{F}})
&\wto\RDer\Sectc(X,\ringtransf_{\Lambda/I[[G/K]]}f_!f^*\cmplx{\sheaf{F}})\\
&\wto\RDer\Sectc(X,{f_{K}}_!f_K^*\ringtransf_{\Lambda/I}\cmplx{\sheaf{F}}).
\end{align*}
Thus, by Proposition~\ref{prop:characterisation of relative
category}, we may assume that $\Lambda$ is a finite ring and that
$G\isomorph H\rtimes\Gamma_{\ell^{\infty}}$ with a finite group $H$. It then suffices to
show that $\RDer\Sectc(X,f_!f^*\cmplx{\sheaf{F}})$ has finite
cohomology groups.

There exists an open
subgroup $\Gamma'\subset \Gamma_{\ell^{\infty}}$ which lies in the centre of $G$.
The scheme $Y_{\Gamma'}$ is separated of finite type over a finite
extension $\FF'$ of $\FF$ and by Lemma~\ref{lem:ismorphisms of G-coverings} the principal covering $(f^{\Gamma'}\colon Y\mto Y_{\Gamma'},\Gamma')$ is isomorphic to the cyclotomic $\Gamma_{\ell^{\infty}}$-covering of $Y_{\Gamma'}$
over $\FF'$.

By Proposition~\ref{prop:ring change for coverings} the
cohomology groups of $\RDer\Sectc(X,f_!f^*\cmplx{\sheaf{F}})$ are
isomorphic to those of
$\RDer\Sectc(Y_{\Gamma'},f^{\Gamma'}_!{f^{\Gamma'}}^*f_{\Gamma'}^*\cmplx{\sheaf{F}})$
which are finite by Proposition~\ref{prop:the Zl case}.
\end{proof}

\begin{cor}\label{cor:Existence of L-function}
Under the same assumptions as above,
$$
\id-\Frob_{\FF}\colon
\RDer\Sectc(\algc{X},f_!f^*\cmplx{\sheaf{F}})\mto\RDer\Sectc(\algc{X},f_!f^*\cmplx{\sheaf{F}})
$$
is a quasi-isomorphism in $w_H\cat{PDG}^{\cont}(\Lambda[[G]])$ and
hence, gives rise to an element
$$
[\id-\Frob_{\FF}]\in\KTh_1(w_H\cat{PDG}^{\cont}(\Lambda[[G]]))
$$
satisfying
$$
d [\id-\Frob_{\FF}]=[\RDer\Sectc(X,f_!f^*\cmplx{\sheaf{F}})]
$$
in $\KTh_0(\cat{PDG}^{\cont,w_H}(\Lambda[[G]]))$.
\end{cor}
\begin{proof}
By Proposition~\ref{prop:fundamental exact sequence}, the cone of $\id-\Frob_{\FF}$ is $\RDer\Sectc(X,f_!f^*\cmplx{\sheaf{F}})$ shifted by one. Hence, $\id-\Frob_{\FF}$ is a quasi-isomorphism in $w_H\cat{PDG}^{\cont}(\Lambda[[G]])$. Theorem~\ref{thm:localisation thm} then implies $d [\id-\Frob_{\FF}]=[\RDer\Sectc(X,f_!f^*\cmplx{\sheaf{F}})]$.
\end{proof}

\begin{defn}
We write $\ncL_G(X/\FF,\cmplx{\sheaf{F}})$ for the inverse of the element $[\id-\Frob_{\FF}]$.
\end{defn}

The element $\ncL_G(X/\FF,\cmplx{\sheaf{F}})$ may be thought of as our analogue of the noncommutative $\ell$-adic $L$-function that is conjectured to exist in \cite{CFKSV}. Note that the assignment $\cmplx{\sheaf{F}}\mapsto\ncL_G(X/\FF,\cmplx{\sheaf{F}})$ extends to a homomorphism
$$
\KTh_0(\cat{PDG}^{\cont}(X,\Lambda))\mto\KTh_1(w_H\cat{PDG}^{\cont}(\Lambda[[G]])).
$$
Moreover, $\ncL_G(X/\FF,\cmplx{\sheaf{F}})$ enjoys the following transformation properties.

\begin{thm}\label{thm:transformation properties}
Consider a separated scheme $X$ of finite type over a finite field $\FF$.
Let $\Lambda$ be any adic $\Int_{\ell}$ algebra and let $\cmplx{\sheaf{F}}$ be a complex in $\cat{PDG}^{\cont}(X,\Lambda)$.
\begin{enumerate}
\item Let $\Lambda'$ be another adic $\Int_{\ell}$-algebra.
For any complex $\cmplx{M}$ in $\Lambda[[G]]^{\op}$-$\cat{SP}(\Lambda')$, we have
$$
\ringtransf_{\cmplx{M[[G]]^{\delta}}}(\ncL_G(X/\FF,\cmplx{\sheaf{F}}))=\ncL_G(X/\FF,\ringtransf_{\cmplx{\widetilde{M}}}\cmplx{\sheaf{F}})
$$
in $\KTh_1(w_H\cat{PDG}^{\cont}(\Lambda'[[G]]))$.

\item Let $H'$ be a closed virtual pro-$\ell$-subgroup of $H$ which is normal in $G$. Then
$$
\ringtransf_{\Lambda[[G/H']]}(\ncL_G(X/\FF,\cmplx{\sheaf{F}}))=\ncL_{G/H'}(X/\FF,\cmplx{\sheaf{F}})
$$
in $\KTh_1(w_{H'}\cat{PDG}^{\cont}(\Lambda[[G/H']]))$.

\item Let $U$ be an open subgroup of $G$ and let $\FF'$ be the finite extension corresponding to the image of $U$ in $\Gamma_{\ell^{\infty}}$. Then
$$
\ringtransf_{\Lambda[[G]]}\big(\ncL_{G}(X/\FF,\cmplx{\sheaf{F}})\big)=\ncL_{U}(Y_U/\FF',f_U^*\cmplx{\sheaf{F}})
$$
in $\KTh_1(w_{H\cap U}\cat{PDG}^{\cont}(\Lambda[[U]]))$.
\end{enumerate}
\end{thm}
\begin{proof}
Assertions (1) and (2) follow from Proposition~\ref{prop:ring change for coverings II} and Proposition~\ref{prop:ring change for coverings}, respectively, in conjunction with Proposition~\ref{prop:example functors} and Proposition~\ref{prop:compatibility of change of rings}. For Assertion (3) we need the following additional reasoning. Consider the commutative diagram
$$
\xymatrix{
Y_U\ar[r]^g\ar[dr]_{f_U}&X\times_\FF \FF'\ar[r]^{a'}\ar[d]_{b'}&\Spec \FF'\ar[d]_b\\
                        &X\ar[r]^{a}                          &\Spec \FF
}
$$
There exists a chain of quasi-isomorphisms
$$
\RDer a_! \RDer {f_U}_!f^U_!{f^U}^*f_U^*\cmplx{\sheaf{F}}\sim b_!\RDer (a'\comp g)_!f^U_!{f^U}^*f_U^*\cmplx{\sheaf{F}}
$$
in $\cat{PDG}^{\cont}(\Spec \FF,\Lambda[[U]])$. Therefore, it suffices to understand the operation of the Frobenius $\Frob_{\FF}$ on $\Sect(\Spec \algc{\FF},b_!\sheaf{F})$ for any \'etale sheaf $\sheaf{F}$ on $\Spec \FF'$. Choosing an embedding of $\FF'$ into $\algc{\FF}$ we obtain an isomorphism
$$
\Sect(\Spec \algc{\FF},b_!\sheaf{F})\isomorph\Sect(\Spec \algc{\FF},\sheaf{F})^{[\FF':\FF]}
$$
under which the Frobenius $\Frob_{\FF}$ on the left-hand side corresponds to multiplication with the matrix
$$
M=
\begin{pmatrix}
0&\hdotsfor{2}&0&\Frob_{\FF'}\\
\id&0&\hdotsfor{2}&0\\
0&\id&0&\hdots&0\\
\vdots&\ddots&\ddots&\ddots&\vdots\\
0&\hdots&0&\id&0
\end{pmatrix}
$$
on the righthand side.
Using only elementary row and column operations one can transform $\id-M$ into the matrix
$$
\begin{pmatrix}
\id& 0 &\hdotsfor{2}&0\\
0&\id & 0&\hdots &0\\
\vdots&\ddots&\ddots&\ddots&\vdots\\
0&\hdots&0&\id& 0\\
0&\hdots&0& 0&\id-\Frob_{\FF'}
\end{pmatrix}.
$$
Since these elementary operations have trivial image in the first $\KTh$-group, we conclude
$$
[\id-\Frob_{\FF}\colon \RDer\Sectc(Y_U\times_{\FF}\algc{\FF},f^U_!{f^U}^*f_U^*\cmplx{\sheaf{F}})]=
[\id-\Frob_{\FF'}\colon \RDer\Sectc(Y_U\times_{\FF'}\algc{\FF},f^U_!{f^U}^*f_U^*\cmplx{\sheaf{F}})]
$$
in $\KTh_1(w_{H\cap U}\cat{PDG}^{\cont}(\Lambda[[U]]))$, from which the assertion follows.
\end{proof}

Let now $\Omega$ be a commutative adic $\Int_{\ell}$-algebra. Consider the set
$$
P=\{f(T)\in\Omega[T]\colon f(0)\in \Omega^{\times}\}
$$
in the polynomial ring $\Omega[T]$. Then $\Omega[T]_P$ is a commutative semilocal ring and the natural homomorphism of $\KTh_1(\Omega[T]_P)=\Omega[T]_P^\times$ to $\KTh_1(\Omega[[T]])=\Omega[[T]]^\times$ is an injection. Furthermore, let $k$ be prime to $\ell$ and
$$
S=\{f\in \Omega[[\Gamma_{k\ell^{\infty}}]]\colon \text{$[f]\in \Omega/\Jac(\Omega)[[\Gamma_{k\ell^{\infty}}]]$ is a nonzerodivisor}\}
$$
be the set in $\Omega[[\Gamma_{k\ell^{\infty}}]]$ that we considered in Lemma~\ref{lem:K_1 in the commutative case}. We write $\gamma$ for the image of the Frobenius $\Frob_{\FF}$ in $\Gamma_{k\ell^{\infty}}$.

\begin{lem}
The homomorphism
$$
\Omega[T]\mto \Omega[[\Gamma_{k\ell^{\infty}}]],\qquad T\mapsto \gamma^{-1},
$$
maps $P$ into $S$.
\end{lem}
\begin{proof}
We can replace $\Omega$ by $\Omega/\Jac(\Omega)$, which is a finite product of finite fields. By considering each component separately, we may assume that $\Omega$ is a finite field. Enlarging $\Omega$ if necessary, we have an isomorphism
$$
\Omega[[\Gamma_{k\ell^{\infty}}]]\isomorph \prod_{\chi\colon \Int/k\Int\mto \Omega^\times}\Omega[[\Gamma_{\ell^{\infty}}]].
$$
Recall that
$$
\Omega[[\Gamma_{\ell^{\infty}}]]\mto \Omega[[T]],\qquad \gamma^{-1}\mapsto T+1,
$$
is an isomorphism. Now it suffices to remark that for any nonzero polynomial $f(T)\in\Omega[T]$ and any $u\in\Omega^\times$, $f(u(T+1))$ is again a nonzero polynomial.
\end{proof}

Extending the classical definition, we define as in \cite{Witte:NoncommutativeLFunctions} the $L$-function for any $\cmplx{\sheaf{F}}$ in $\cat{PDG}^{\cont}(X,\Omega)$ as the element
$$
L(\cmplx{\sheaf{F}},T)=\prod_{x\in X^0}[\id-T^{\deg(x)}\Frob_{k(x)}\colon \ringtransf_{\Omega[[T]]}(\cmplx{\sheaf{F}}_\xi)]^{-1}\in \KTh_1(\Omega[[T]])\isomorph \Omega[[T]]^{\times}.
$$
Here, the product extends over the set $X^0$ of closed points of $X$,
$$
\Frob_{k(x)}=\Frob_{\FF}^{\deg(x)}\in \Gal(\algc{k(x)}/k(x))
$$
denotes the geometric Frobenius of the residue field $k(x)$, and $\xi$ is a geometric point over $x$.

We are ready to establish the link between the classical $L$-function and the element $\ncL_G(X/\FF,\cmplx{\sheaf{F}})$. For this, it is essential to impose the additional condition that $\ell$ is different from the characteristic of $\FF$.

\begin{thm}\label{thm:link to classical L-function}
Let $X$ be a separated scheme in $\cat{Sch}_\FF$, let $\ell$ be different from the characteristic of $\FF$, and let $(f\colon Y\mto X,G)$ be an $\ell$-admissible principal covering containing the cyclotomic $\Gamma_{k\ell^{\infty}}$-covering.  Furthermore, let $\Lambda$ and $\Omega$ be adic $\Int_{\ell}$-algebras with $\Omega$ commutative. For every $\cmplx{\sheaf{F}}$ in $\cat{PDG}^{\cont}(X,\Lambda)$ and every $\cmplx{M}$ in $\Lambda[[G]]^\op$-$\cat{SP}(\Omega)$, we have
$L(\ringtransf_{\cmplx{\widetilde{M}}}(\cmplx{\sheaf{F}}),T)\in\KTh_1(\Omega[T]_P)$ and
$$
\ringtransf_{\Omega[[\Gamma_{k\ell^{\infty}}]]}\ringtransf_{\cmplx{M[[G]]^{\delta}}}\big(\ncL_{G}(X/\FF,\cmplx{\sheaf{F}})\big)= L(\ringtransf_{\cmplx{\widetilde{M}}}(\cmplx{\sheaf{F}}),\gamma^{-1})
$$
in $\KTh_1(\Omega[[\Gamma_{k\ell^{\infty}}]]_S)$.
\end{thm}
\begin{proof}
By Theorem~\ref{thm:transformation properties} it suffices to consider the case $G=\Gamma_{k\ell^{\infty}}$, $\Lambda=\Omega$, and $\cmplx{M}=\Omega$. Let $p\cat{SP}(\Omega[T])$ be the Waldhausen category of strictly perfect complexes of $\Omega[T]$-modules with quasi-isomorphisms being the morphisms which become quasi-isomorphisms in $\cat{SP}(\Omega[[T]])$, that means, precisely those whose cone has $P$-torsion cohomology groups. Then
$$
\KTh_n(p\cat{SP}(\Omega[T]))=\KTh_n(\Omega[T]_P)
$$
for $n\geq 1$ according to \cite{WeibYao:Localization}. It is easy to show that there exists a strictly perfect complex of $\Omega$-modules $\cmplx{Q}$ with an endomorphism $f$ and a quasi-isomorphism $q\colon \cmplx{Q}\mto \RDer\Sectc(\algc{X},\cmplx{\sheaf{F}})$ such that the following diagram commutes up to homotopy:
$$
\xymatrix
{
\cmplx{Q}\ar[r]^{q}\ar[d]^{f}&\RDer\Sectc(\algc{X},\cmplx{\sheaf{F}})\ar[d]^{\Frob_{\FF}}\\
\cmplx{Q}\ar[r]^{q}&\RDer\Sectc(\algc{X},\cmplx{\sheaf{F}})
}
$$
(see e.\,g.\ \cite[Lemma~3.3.2]{Witte:PhD}). By the Grothendieck trace formula \cite[Fonction $L$ mod $\ell^n$, Theorem~2.2]{SGA4h} (or also \cite[Theorem~7.2]{Witte:NoncommutativeLFunctions}) we know that
$$
L(\cmplx{\sheaf{F}},T)=[\id-T\Frob_{\FF}\colon\ringtransf_{\Omega[[T]]}\RDer\Sectc(\algc{X},\cmplx{\sheaf{F}})]^{-1}
$$
in $\KTh_1(\Omega[[T]])$.
Hence, the above homotopy-commutative diagram implies
$$
L(\cmplx{\sheaf{F}},T)=[\id-Tf\colon \ringtransf_{\Omega[[T]]}\cmplx{Q}]^{-1}
$$
in $\KTh_1(\Omega[[T]])$ and by Proposition~\ref{prop:twisted Frob operation} also
$$
\ncL_{\Gamma_{k\ell^{\infty}}}(X/\FF,\cmplx{\sheaf{F}})=[\id-\gamma^{-1}f\colon \cmplx{Q}]^{-1}
$$
in $\KTh_1(\Omega[[\Gamma_{k\ell^{\infty}}]]_S)$. Hence, $L(\cmplx{\sheaf{F}},T)$ and $\ncL_{\Gamma_{k\ell^{\infty}}}(X/\FF,\cmplx{\sheaf{F}})$ are the images of the element $[\id-Tf\colon \Omega[T]\tensor_{\Omega}\cmplx{Q}]^{-1}$ under the homomorphisms
$$
\KTh_1(\Omega[T]_P)\mto \KTh_1(\Omega[[T]]),\qquad \KTh_1(\Omega[T]_P)\mto \KTh_1(\Omega[[\Gamma_{k\ell^{\infty}}]]_S),
$$
respectively.
\end{proof}

\begin{rem}
If $\Omega$ is noncommutative we do not expect that
$$
\KTh_1(p\cat{SP}(\Omega[T]))\mto \KTh_1(\Omega[[T]])
$$
is always injective. However, the construction in the above proof identifies a canonical preimage of $L(\cmplx{\sheaf{F}},T)$ in the group $\KTh_1(p\cat{SP}(\Omega[T]))$. Indeed, one checks that $[\id-\gamma^{-1}f\colon \cmplx{Q}]$ does not depend on the particular choice of $\cmplx{Q}$ and $f$. This preimage can then be used to extend Theorem~\ref{thm:link to classical L-function} to noncommutative rings $\Omega$.
\end{rem}

Theorem~\ref{thm:special case NCIMC} in the introduction is easily seen to be a special case of Theorem~\ref{thm:S-torsion property}, Corollary~\ref{cor:Existence of L-function}, and Theorem~\ref{thm:link to classical L-function} for $\Lambda=\Int_{\ell}$ and $(f\colon Y\mto X,G)$ being a subcovering of the universal covering of a connected scheme $X$.

\begin{rem}
As shown in \cite[Cor.~3.3]{Witte:Splitting}, one can construct for $i\geq 0$ a section
$$
s\colon \KTh_i(\cat{PDG}^{\cont,w_H}(\Lambda[[G]]))\mto\KTh_{i+1}(w_H\cat{PDG}^{\cont}(\Lambda[[G]]))$$
 of the boundary homomorphism $d$. We can use $s$ to define an ``algebraic $L$-function''
$$
\ncL^{\mathrm{alg}}_G(\cmplx{M})=s(\cmplx{M})\in \KTh_1(w_H\cat{PDG}^{\cont}(\Lambda[[G]]))
$$
for any complex $\cmplx{M}\in \cat{PDG}^{\cont,w_H}(\Lambda[[G]])$, extending the definition of Burns in \cite{Burns:AlgebraicLfunctions}. The quotient
$
u(\cmplx{\sheaf{F}})=\ncL_G(\cmplx{\sheaf{F}})/\ncL^{\mathrm{alg}}_G(\RDer\Sectc(X,\cmplx{\sheaf{F}}))
$
is then a well-defined element of $\KTh_1(\Lambda[[G]])$. In general, $u(\cmplx{\sheaf{F}})$ will not be trivial. If for example $X=\Spec \FF$ with a finite field $\FF$ with $q$ elements, $\ell$ is a prime not dividing $q(q-1)$, $G=\Gamma_{\ell^{\infty}}$, $\cmplx{\sheaf{F}}=\Int_{\ell}(1)$, then $\ncL_G(\Int_{\ell}(1))=(1-q^{-1}\gamma^{-1})^{-1}$ is a unit in $\Int_{\ell}[[\Gamma_{\ell^{\infty}}]]$. This implies that the class of $\RDer\Sectc(\Spec \FF,\Int_{\ell}(1))$ in
\begin{align*}
\KTh_0(\cat{PDG}^{\cont,w_H}(\Int_{\ell}[[\Gamma_{\ell^{\infty}}]]))&=
\KTh_0(\Int_{\ell}[[\Gamma_{\ell^{\infty}}]],\Int_{\ell}[[\Gamma_{\ell^{\infty}}]]_{(\ell)})\\
&=
\Int_{\ell}[[\Gamma_{\ell^{\infty}}]]_{(\ell)}^{\times}/\Int_{\ell}[[\Gamma_{\ell^{\infty}}]]^{\times}
\end{align*}
is trivial and hence, the algebraic $L$-function $\ncL^{\mathrm{alg}}_G(\RDer\Sectc(\Spec \FF,\Int_{\ell}(1)))$ must be trivial, too. Thus, $u(\Int_{\ell}(1))=(1-q^{-1}\gamma^{-1})^{-1}\neq 1$ in this case. It would be interesting to find more explicit descriptions of $u(\cmplx{\sheaf{F}})$ and of its image under the natural projection map $\KTh_1(\Lambda[[G]])\mto \KTh_1(\Lambda)$.
\end{rem}

\section*{Appendix}

Let $\cat{wW}$ be a Waldhausen category with weak equivalences
$\cat{w}$, and let $\cat{vW}$ be the same category with the same
notion of cofibrations, but with a coarser notion
$\cat{v}\subset\cat{w}$ of weak equivalences. We assume that
$\cat{wW}$ is \emph{saturated} and \emph{extensional}, i.\,e.\
\begin{enumerate}
\item if $f$ and $g$ are composable and any two of the morphisms
$f$, $g$ and $g\comp f$ are in $\cat{w}$, then so is the third;

\item if the two outer components of a morphism of exact sequences
in $\cat{wW}$ are in $\cat{w}$, then so is the middle one.
\end{enumerate}

We denote by $\cat{vW^w}$ the full subcategory of $\cat{vW}$
consisting of those objects $A$ such that $0\mto A$ is in
$\cat{w}$. With the notions of cofibrations and weak equivalences
in $\cat{vW}$, this subcategory is again a Waldhausen category.

Under the additional assumption that there exists an appropriate
notion of cylinder functors in $\cat{vW}$ and $\cat{wW}$, which we
will explain below, Waldhausen's localisation theorem
\cite{ThTr:HAKTS+DC}, Theorem~1.8.2, states that the natural
inclusion functors $\cat{vW^w}\mto\cat{vW}\mto\cat{wW}$ induce a
homotopy fibre sequence of the associated $\KTh$-theory spaces and
hence, a long exact sequence
\begin{align*}
\dotso&\mto \KTh_n(\cat{vW^w})\mto \KTh_{n}(\cat{vW})\mto
\KTh_{n}(\cat{wW})\xrightarrow{d}
\KTh_{n-1}(\cat{vW^w})\mto \dotso\\
\dotso&\mto \KTh_1(\cat{vW})\mto \KTh_1(\cat{wW})\xrightarrow{d}
\KTh_0(\cat{vW^w})\mto \KTh_0(\cat{vW})\mto \KTh_0(\cat{wW})\mto
0.
\end{align*}
In this appendix, we will give an explicit description of the
connecting homomorphism
$d\colon\KTh_1(\cat{wW})\mto\KTh_0(\cat{vW^w})$ in terms of the
1-types of the Waldhausen categories, as defined in \cite{MT:1TWKTS}.
A similar description has also been derived in \cite[Theorem~4.1]{Staffeldt} (up to some obvious sign errors) using more sophisticated arguments.

We begin by recalling the definition of a cylinder functor. For
any Waldhausen category $\cat{W}$, the \emph{category of
morphisms} $\cat{Mor}(\cat{W})$ is again a Waldhausen category
with the following cofibrations and weak equivalences. A morphism
$\alpha\mto \beta$ in $\cat{Mor}(\cat{W})$, i.\,e.\ a commutative
square
$$
\xymatrix{ A \ar[r]^{\epsilon}\ar[d]_{\alpha}&B\ar[d]^{\beta}\\
A'\phantom{,}\ar[r]_{\epsilon'}&B',}
$$
is a cofibration if both $\epsilon$ and $\epsilon'$ are
cofibrations. It is a weak equivalence if both $\epsilon$ and
$\epsilon'$ are weak equivalences. One checks easily that the
functors
\begin{align*}
s\colon \cat{Mor}(\cat{W}) \mto \cat{W},&\qquad
(A\xrightarrow{\alpha}
A')\mapsto A,\\
t\colon \cat{Mor}(\cat{W}) \mto \cat{W},&\qquad
(A\xrightarrow{\alpha}
A')\mapsto A',\\
S\colon \cat{W}\mto \cat{Mor}(\cat{W}),&\qquad A\mapsto
 (A\mto 0),\\
T\colon\cat{W}\mto \cat{Mor}(\cat{W}),&\qquad A\mapsto (0\mto A)
\end{align*}
are Waldhausen exact. Let $ar\colon s\mto t$ denote the natural
transformation given by $ar(\alpha)=\alpha$.

For two Waldhausen categories $\cat{W}_1$ and $\cat{W}_2$ we let
$$
\Fun(\cat{W}_1,\cat{W}_2)
$$
denote the \emph{Waldhausen category of exact functors} with
natural transformations as morphisms. A natural transformation
$\alpha\colon F\mto G$ is a cofibration if
\begin{enumerate}
\item for each object $C$ in $\cat{W}_1$, the morphism
$\alpha(C)\colon F(C)\mto G(C)$ is a cofibration,

\item for each cofibration $C\cto C'$ in $\cat{W}_1$, $G(C)\cup_{F(C)}F(C')\mto G(C')$ is a cofibration.
\end{enumerate}
A natural transformation $\alpha\colon F\mto G$ is a weak
equivalence if for each object $C$ in $\cat{W}_1$, the morphism
$\alpha(C)\colon F(C)\mto G(C)$ is a weak equivalence.

\begin{defna}\label{defn:cylinder functor}
A \emph{cylinder functor} for $\cat{W}$ is an exact functor
$$
\Cyl\colon \cat{Mor}(\cat{W})\mto \cat{W}
$$
together with natural transformations $j_1\colon s\mto \Cyl$,
$j_2\colon t\mto \Cyl$, $p\colon \Cyl \mto t$ such that
\begin{enumerate}
\item $p\comp j_1=ar$, $p\comp j_2=\id$,

\item $j_1\oplus j_2\colon s\oplus t\cto \Cyl$ is a cofibration in
$\Fun(\cat{Mor}(\cat{W}),\cat{W})$,

\item $\Cyl\comp T=\id$ and the compositions of $j_2$ with $T$ and
$p$ with $T$ are the identity transformation on $\id$.
\end{enumerate}
A cylinder functor satisfies the \emph{cylinder axiom} if
\begin{enumerate}
\item[(4)] $p\colon \Cyl\wto t$ is a weak equivalence in
$\Fun(\cat{Mor}(\cat{W}),\cat{W})$.
\end{enumerate}
\end{defna}

\begin{rema}
The above definition of a cylinder functor is clearly equivalent
to the one given in \cite{Wal:AlgKTheo}, Definition~1.6. Thomason
claims that it is also equivalent to the one given in
\cite{ThTr:HAKTS+DC}, Definition~1.3.1. However, it seems at least
not to be completely evident from the axioms stated there that
$\Cyl$ preserves pushouts along cofibrations.
\end{rema}

We further set
\begin{align*}
\Cone=\Cyl/s&\colon \cat{Mor}(\cat{W})\mto \cat{W},&(A\xrightarrow{\alpha} A')&\mapsto \Cyl(\alpha)/A,  \\
\Sigma=\Cone\comp S&\colon \cat{W}\mto \cat{W},&A&\mapsto
\Cone(A\mto 0).
\end{align*}
Note that $t\cto \Cone \qto \Sigma\comp s$ is an exact sequence in
$\Fun(\cat{Mor}(\cat{W}),\cat{W})$
for any cylinder functor $\Cyl$.

\begin{defna}\label{defn:stable quadratic module}
A \emph{stable quadratic module} $M_*$ is a homomorphism of groups $\del_M\colon M_1\mto M_0$ together with
a pairing
$$
\pair{-}{-}\colon M_0\times M_0\mto M_1
$$
satisfying the following identities for any $a,b\in M_1$ and
$X,Y,Z\in M_0$:
\begin{enumerate}
    \item $\pair{\del_M a}{\del_M b}=\com{b}{a},$
    \item $\del_M\pair{X}{Y}=\com{Y}{X},$
    \item $\pair{X}{Y}\pair{Y}{X}=1,$
    \item $\pair{X}{YZ}=\pair{X}{Y}\pair{X}{Z}.$
\end{enumerate}
\end{defna}

We set $a^X=a\pair{X}{\del a}$ for $a\in M_1$, $X\in M_0$. Note that this defines a right action of $M_0$ on $M_1$. Furthermore, we let
$$
\hg_1(M_*)=\ker \del_M,\qquad \hg_0(M_*)=\coker \del_M
$$
denote the \emph{homotopy groups} of $M_*$.

Assume that $f\colon M_*\mto N_*$ is any morphism
of stable quadratic modules such that $f_0$ is injective. Then $f_1(N_1)$
is a normal subgroup of $\del_N^{-1}(f_0(N_0))$. We set
$$
\hg_0(M_*,N_*)=\del_N^{-1}(f_0(N_0))/f_1(N_1)
$$
and obtain an exact sequence
$$
\hg_1(M_*)\mto
\hg_1(N_*)\mto\hg_0(M_*,N_*)\mto\hg_0(M_*)\mto\hg_0(N_*).
$$

Muro and Tonks give the following definition of the $1$-type of a Waldhausen category \cite[Definition~1.2]{MT:1TWKTS}.

\begin{defna}\label{defn:1-type}
Let $\cat{W}$ be a Waldhausen category. The \emph{algebraic 1-type}
$\D_*\cat{W}$ of $\cat{W}$ is the stable quadratic module generated by
\begin{itemize}
    \item[(G0)] the symbols $[X]$ for each object $X$ in $\cat{W}$ in degree $0$,
    \item[(G1)] the symbols $[w]$ and $[\Delta]$ for each weak
    equivalence $w$ and each exact sequence $\Delta$ in $\cat{W}$,
\end{itemize}
with $\del$ given by
\begin{itemize}
    \item[(R1)] $\del[\alpha]=[B]^{-1}[A]$ for $\alpha\colon A\wto
    B$,
    \item[(R2)] $\del[\Delta]=[B]^{-1}[C][A]$ for $\Delta\colon A\cto B\qto
    C$.
\end{itemize}
and
\begin{itemize}
    \item[(R3)] $\pair{[A]}{[B]}=[B\cto
    A\oplus B \qto A]^{-1}[A\cto A\oplus B\qto B]$ for any pair of
    objects $A, B$.
\end{itemize}
Moreover, we impose the following relations:
\begin{itemize}
    \item[(R4)] $[0\cto 0\qto 0]=1_{\D_1}$,
    \item[(R5)] $[\beta\alpha]=[\beta][\alpha]$ for $\alpha\colon A\wto
    B$, $\beta\colon B\wto C$,
    \item[(R6)] $[\Delta'][\alpha][\gamma]^{[A]}=[\beta][\Delta]$ for any commutative diagram
    $$
    \xymatrix{
    A\lbl{\Delta\colon}\ar@{ >->}[r]\ar[d]^{\sim}_{\alpha}&
    B\ar@{>>}[r]\ar[d]^{\sim}_{\beta}
    &C\ar[d]^{\sim}_{\gamma}\\
    A'\lbl{\Delta'\colon}\ar@{ >->}[r]&B'\ar@{>>}[r]&C'}
    $$
    \item[(R7)] $[\Gamma_1][\Delta_1]=[\Delta_2][\Gamma_2]^{[A]}$ for any
    commutative diagram
    $$
    \xymatrix{A\lbl{\Delta_1\colon}\ar@{ >->}[r]\ar@{=}[d]
    &B\lbu{\Gamma_1\colon}\ar@{>>}[r]\ar@{ >->}[d]
    &C\lbu{\Gamma_2\colon}\ar@{ >->}[d]\\
    A\lbl{\Delta_2\colon}\ar@{ >->}[r]\ar@{>>}[d]&D\ar@{>>}[r]\ar@{>>}[d]&E\ar@{>>}[d]\\
    {0}\ar@{ >->}[r]&F\ar@{=}[r]&F}
    $$
\end{itemize}
\end{defna}

Muro and Tonks then prove that
$$
\KTh_1(\cat{W})=\hg_1(\D_*(\cat{W})),\qquad \KTh_0(\cat{W})=\hg_0(\D_*(\cat{W})).
$$

The following theorem gives our explicit description of the connecting homomorphism.

\begin{thma}\label{thm:localisation thm}
Let $\cat{wW}$ be a Waldhausen category and $\cat{vW}$ the same
category with a coarser notion of weak equivalences. Assume that
$\cat{wW}$ is saturated and extensional and let $\Cyl$ be a
cylinder functor for both $\cat{wW}$ and $\cat{vW}$ which
satisfies the cylinder axiom for $\cat{wW}$. Then the assignment
\begin{align*}
d(\Delta)&=1 &&\text{for every exact sequence $\Delta$ in
$\cat{wW}$,}\\
d(\alpha)&=[\Cone(\alpha)]^{-1}[\Cone(\id_A)] &&\text{for every
weak equivalence $\alpha\colon A\mto A'$ in $\cat{wW}$}
\end{align*}
defines a homomorphism $d\colon \D_1(\cat{wW})\mto
K_0(\cat{vW^w})$ and the sequence
$$
\KTh_1(\cat{vW})\mto \KTh_1(\cat{wW})\xrightarrow{d}
\KTh_0(\cat{vW^w})\mto \KTh_0(\cat{vW})\mto \KTh_0(\cat{wW})\mto 0
$$
is exact.
\end{thma}
\begin{proof}
We may view $\KTh_0(\cat{vW^w})$ as a stable quadratic module with trivial
group in degree zero and $\KTh_0(\cat{vW^w})$ in degree one. By
the universal property of $\D_*(\cat{wW})$ it suffices to verify
the following two assertions in order to show that the
homomorphism $d\colon \D_1(\cat{wW})\mto\KTh_0(\cat{vW^w})$ is
well-defined.
\begin{enumerate}
\item For commutative diagrams
$$
    \xymatrix{
    A\ar@{ >->}[r]\ar[d]^{\sim}_{\alpha}&
    B\ar@{>>}[r]\ar[d]^{\sim}_{\beta}
    &C\ar[d]^{\sim}_{\gamma}\\
    A'\ar@{ >->}[r]&B'\ar@{>>}[r]&C'}
$$
in $\cat{wW}$ we have $d(\beta)=d(\alpha)d(\gamma)$.
\item For weak equivalences $\alpha\colon A\wto
    B$, $\beta\colon B\wto C$ in $\cat{wW}$ we have $d(\beta\comp\alpha)=d(\beta)d(\alpha)$.
\end{enumerate}

Assertion $(1)$ follows easily by applying the exact functors
$\Cone$ and $\alpha\mapsto\Cone(\id_{s(\alpha)})$ to the exact sequence
$\alpha\cto\beta\qto\gamma$ in $\cat{Mor}(\cat{wW})$. We prove
Assertion $(2)$.

First, we consider a weak equivalence $\alpha\colon A\wto B$ in
$\cat{wW}$ between objects $A$ and $B$  in $\cat{vW^w}$. The exact
sequences
\begin{align*}
B\cto \Cone(\alpha)\qto \Sigma A,\\
B\cto \Cone(\id_B)\qto \Sigma B
\end{align*}
in $\cat{vW^w}$ imply
\begin{align*}
d(\alpha)d(B\wto 0)&=[\Cone(\alpha)]^{-1}[\Cone(\id_A)][\Sigma
B]^{-1}[\Cone(\id_B)]\\
&=[\Sigma A]^{-1}[\Cone(\id_A)]\\
&=d(A\wto 0).
\end{align*}
We obtain
$$
d(\beta\comp\alpha)=d(C\wto0)^{-1}d(A\wto 0)=d(\beta)d(\alpha)
$$
in the special case that $\alpha\colon A\wto B$ and $\beta\colon
B\wto C$ are weak equivalences in $\cat{wW}$ between objects $A$,
$B$, and $C$ in $\cat{vW^w}$.

Let now $\alpha\colon A\wto B$ and $\beta\colon B\wto C$ by
arbitrary weak equivalences in $\cat{wW}$. Viewing the vertical
morphisms in the commutative diagram
$$
\xymatrix{
A\ar@{=}[r]\ar@{=}[d]&A\ar@{=}[r]\ar[d]^{\sim}_{\alpha}&A\ar[d]^{\sim}_{\beta\comp\alpha}\\
A\ar[r]^{\sim}_{\alpha}&B\ar[r]^{\sim}_{\beta}         &C}
$$
as morphisms $\id_A\wto \alpha\wto \beta\comp\alpha$ in
$\cat{Mor}(\cat{wW})$ and applying the exact sequence $t\cto \Cone
\qto \Sigma s$ we obtain the following commutative diagram with
exact rows.
$$
\xymatrix{
    A\ar@{ >->}[r]\ar[d]^{\sim}_{\alpha}&
    \Cone(\id_A)\ar@{>>}[r]\ar[d]^{\sim}_{\alpha_*}
    &\Sigma A\ar@{=}[d]\\
    B\ar@{ >->}[r]\ar[d]^{\sim}_{\beta}&\Cone(\alpha)\ar@{>>}[r]\ar[d]^{\sim}_{\beta_*}&
    \Sigma A\ar@{=}[d]\\
    C\ar@{ >->}[r]&\Cone(\beta\comp\alpha)\ar@{>>}[r]&\Sigma A
    }
$$
Assertion $(1)$ and the previously proved special case of
Assertion $(2)$ imply
$$
d(\beta\comp\alpha)=d(\beta_*\comp\alpha_*)=d(\beta_*)d(\alpha_*)=d(\beta)d(\alpha).
$$

This completes the proof of Assertion $(2)$ in general. Hence, we
have also proved the existence of the homomorphism $d\colon
\D_1(\cat{wW})\mto \KTh_0(\cat{vW^w})$.

We will now prove the exactness of the sequence in the statement
of the theorem. Note that $\D_0(\cat{vW^w})$ injects into
$\D_0(\cat{vW})$ and that $\D_0(\cat{vW})=\D_0(\cat{wW})$.
Write $K=\hg_0(\D_*(\cat{vW}),\D_*(\cat{wW}))$, i.\,e.\ $K$ is
the cokernel of the natural homomorphism
$\D_1(\cat{vW})\mto\D_1(\cat{wW})$. As explained above, the
sequence of abelian groups
$$
\KTh_1(\cat{vW})\mto \KTh_1(\cat{wW})\mto K\mto
\KTh_0(\cat{vW})\mto \KTh_0(\cat{wW})\mto 0
$$
is exact.

Let $\alpha\colon A\wto B$ be a weak equivalence in $\cat{vW}$.
Since $\Cone\colon \cat{Mor}(\cat{vW})\mto\cat{vW}$ is an exact
functor, we see that the induced morphism $\alpha_*\colon
\Cone(\id_A)\wto \Cone(\alpha)$ is a weak equivalence in
$\cat{vW^w}$. Hence,
$$
d(\alpha)=d(\alpha_*)=1,
$$
i.\,e.\ the homomorphism $d$ factors through $K$. It remains to
show that $d\colon K\mto \KTh_0(\cat{vW^w})$ is an isomorphism.

Consider the homomorphism $h\colon \D_0(\cat{vW^w})\mto
\D_1(\cat{wW})$ induced by sending an object $X$ of $\cat{vW^w}$
to $[X\wto 0]$ in $\D_1(\cat{wW})$. One checks easily that
$h\comp\del_{\D_*(\cat{vW^w})}$ agrees with the natural
homomorphism $\D_1(\cat{vW^w})\mto\D_1(\cat{wW})$; hence, $h$
induces a homomorphism $H\colon \KTh_0(\cat{vW^w})\mto K$.

For any weak equivalence $\alpha\colon A\wto B$ in $\cat{wW}$ we
have
\begin{align*}
H(d(\alpha))&=[\Cone(\alpha)\wto 0]^{-1}[\Cone(\id_A)\wto 0]\\
&=[\Cone(\id_A)\xrightarrow{\alpha_*}\Cone(\alpha)]\\
&=[\alpha];
\end{align*}
for any object $X$ in $\cat{vW^w}$ we have
$$
d(H(X))=[\Sigma X]^{-1}[\Cone(\id_X)]=[X].
$$
Therefore, $d\colon K\mto \KTh_0(\cat{vW^w})$ is indeed an
isomorphism with inverse $H$.
\end{proof}

Note that if $\Cyl$ also satisfies the cylinder axiom in $\cat{vW}$, then $[\Cone(\id_A)]=1$ in $\KTh_0(\cat{vW^w})$ for every object $A$ in $\cat{wW}$. This will be the case in the most common situations.

Assume that $S$ is a left denominator set in a ring $R$ and let $\cat{vW}$ be the category of strictly perfect complexes of left $R$-modules with the class $\cat{v}$ of quasi-isomorphisms as weak equivalences. Let further $\cat{w}$ be the class of complex morphisms which become quasi-isomorphisms after localisation with respect to $S$. Then the usual cylinder functor $\Cyl$ satisfies the cylinder axiom in $\cat{vW}$ and $\cat{wW}$ and the resulting long exact localisation sequence identifies with the localisation sequence
\begin{align*}
\dotso&\mto \KTh_n(R,S^{-1}R)\mto \KTh_{n}(R)\mto
\KTh_{n}(S^{-1}R)\xrightarrow{d}
\KTh_{n-1}(R,S^{-1}R)\mto \dotso\\
\dotso&\mto \KTh_1(R)\mto \KTh_1(S^{-1}R)\xrightarrow{d}
\KTh_0(R,S^{-1}R)\mto \KTh_0(R)\mto \KTh_0(S^{-1}R)
\end{align*}
\cite{WeibYao:Localization}. It is then easy to see that the boundary homomorphism constructed in Theorem~\ref{thm:localisation thm} satifies the formula stated in \cite[p.~2]{WeibYao:Localization} and hence, agrees with the classical boundary homomorphism in this situation.

The above theorem can also be applied to derive the description of Weiss' generalised Whitehead torsion given in \cite[Remark~6.3]{Muro:Maltsiniotis}.